\documentclass[11pt]{amsart}
\usepackage{amsmath, amsthm, amscd, amssymb,amsfonts}
 \usepackage{color}
 \usepackage{mathrsfs}
 \usepackage[all]{xy}
 \usepackage{hyperref}

\numberwithin{equation}{section}

\newtheorem{thm}{Theorem}[section]
\newtheorem*{thmA}{Theorem A}
\newtheorem*{thmB}{Theorem B}

\newtheorem{cor}[thm]{Corollary}
\newtheorem*{cor*}{Corollary}
\newtheorem{lem}[thm]{Lemma}

\newtheorem*{con*}{Conjecture}
\newtheorem*{prob*}{Problem}

\theoremstyle{definition}
\newtheorem{defn}[thm]{Definition}

\theoremstyle{remark}
\newtheorem{rem}[thm]{Remark}

\renewcommand{\r}{\gamma}

\newcommand{\ds}{\displaystyle}

\newcommand{\bbN}{\mathbb{N}}

\begin{document}
\title{Topological Rigidity for FJ  by the infinite cyclic group}

\author[Kun Wang]{Kun Wang}
\address{Department of Mathematics,
Vanderbilt University,
1326 Stevenson Center,
Nashville, TN 37240}
\email{kun.wang@vanderbilt.edu}

\begin{abstract}

We call a group FJ if it satisfies the $K$- and $L$-theoretic Farrell-Jones conjecture 
with coefficients in $\mathbb Z$.  We show that if $G$ is FJ, then the simple Borel conjecture (in dimensions $\ge 5$) holds for every group of the form $G\rtimes\mathbb Z$.  If in addition $Wh(G\times \mathbb Z)=0$, which is true for all known torsion free FJ groups,  then the bordism Borel conjecture (in dimensions $n\ge 5$) holds for $G\rtimes\mathbb Z$.  %We also show that if the $L$-theoretic Farrell-Jones conjecture with coefficients in $\mathbb Z$ holds for a torsion free group $G$, then the Novikov conjecture holds for any repeated semi-direct product  $\big(((G\rtimes\mathbb Z)\rtimes\mathbb Z)\cdots\big)\rtimes\mathbb Z$.  
One of the key ingredients in proving these rigidity results is another main result,  which says that if a torsion free group $G$ satisfies the $L$-theoretic Farrell-Jones conjecture with coefficients in $\mathbb Z$, then any semi-direct product $G\rtimes\mathbb Z$ also satisfies the $L$-theoretic Farrell-Jones conjecture with coefficients in $\mathbb Z$. Our result is indeed more general and implies the $L$-theoretic Farrell-Jones conjecture with coefficients in additive categories is closed under extensions of torsion free groups. This enables us to extend the class of groups which satisfy the Novikov conjecture.% We also obtain an obstruction for the corresponding statement to hold for the $K$-theoretic Farrell-Jone conjecture. 

\end{abstract}

\maketitle
\section{Introduction}

In the area of manifold topology, one of the most intriguing unsolved problem is the Borel conjecture on topological rigidity of closed aspherical manifolds.  Recall that a manifold is called \textit{aspherical} if its universal cover is contractible.
\vskip 5pt
\noindent
\textit{The Borel Conjecture}. Every closed aspherical manifold is topologically rigid.  That is, every homotopy equivalence 
between any two closed aspherical manifolds is homotopic to a homeomorphism.
\vskip 5pt

Note that every homeomorphism between two closed manifolds is simple, i.e.  its Whitehead torsion vanishes. This follows from a theorem of 
Chapman \cite{Chap} which says that every homeomorphism between two finite $CW$-complexes is simple.   Therefore, it is reasonable to study the simple version of the Borel conjecture, which is the converse 
(up to homotopy) to Chapman's topological invariance theorem of Whitehead torsion for closed aspherical manifolds: 

\vskip 5pt
\noindent
\textit{The Simple Borel Conjecture.} Every closed aspherical manifold is simply  topologically rigid. That is, every simple homotopy equivalence between any two closed aspherical manifolds is homotopic to a homeomorphism.

\vskip 5pt
Clearly, the simple Borel conjecture is simpler than the Borel conjecture. The passage from the simple Borel conjecture to the Borel conjecture is the famous conjecture which states that the Whitehead group of any torsion free group vanishes (note that the fundamental group of a closed aspherical manifold is  torsion free).  
%in homotopy theory, which states that every homotopy equivalence between two finite aspherical $CW$-complexes  is homotopic to a simple homotopy equivalence. This conjecture is sometimes known as the \textit{Whitehead conjecture}.

Another type of rigidity question one can ask for aspherical  manifold is the following bordism type rigidity:

\vskip 5pt
\noindent
\textit{The Bordism Borel Conjecture:} Every closed aspherical manifold is bordismly topologically  rigid. That is,  every homotopy equivalence $f: N\longrightarrow M$ from another closed manifold $N$ to $M$ is $h$-cobordant to the identity $1: M\longrightarrow M$. 

\vskip 5pt
Recall that a cobordism $(W; M_1, M_2)$ is called an \textit{$h$-corbordim} if the inclusions $M_1\hookrightarrow W,\ M_2\hookrightarrow W$ are homotopy equivalence. Two homotopy equivalences $f_i: M_i\longrightarrow X, i=1, 2$ are called \textit{h-cobordant} if there is  an $h$-cobordism $(W; M_1, M_2)$ and a homotopy equivalence $(F; f_1, f_2): (W; M_1, M_2)\longrightarrow (X\times[0, 1]; X\times{0},  X\times{1})$. Note that, if we require the map $f$ to be simple and $h$-cobordant to be $s$-cobordant in the bordism Borel conjecture, we then get the $s$-cobordism version of the Borel conjecture, which by the $s$-cobordism theorem,  is just the simple Borel conjecture in dimensions $\ge 5$.  

From the viewpoint of surgery theory, the simple version and the bordism version of the Borel conjecture are very natural, since they have direct connection to the surgery long exact sequences. We refer the reader to \cite{Luck2} for a survey on aspherical manifolds and \cite{FRR} for a discussion about the above three versions of the Borel conjecture.

For convenience, we introduce the following. We say the Borel conjecture, or the simple Borel conjecture, or the bordism Borel conjecture, holds for a group $G$ if either $G$ cannot be realized as the fundamental group of a closed aspherical manifold, or the respective conjecture holds for every closed aspherical manifold  with fundamental group isomorphic to $G$. It is an interesting problem to identify which groups  can be realized as the fundamental group of a closed aspherical manifold. Conjecturally, a group has this property if and only if the group is a finitely presented Poincar\'e duality group. This conjecture is usually referred to as  Wall's conjecture (though Wall \cite{Wall3} did not include the finitely presented condition, which would otherwise make the conjecture false). See \cite{DM} for a survey.

In recent years, there has been significant progress on the proof of the Borel conjecture for a large class of groups,  due to the solutions of the Farrell-Jones conjecture (FJC for short) for these groups.  See the works by many authors \cite{BFL},\cite{BL2},\cite{BLR2},\cite{BLRR},\cite{WC},\cite{WC1},\cite{KLR},\cite{GMR},\cite{FW1},\cite{FW2},\cite{FW3},\cite{RH}. 
Roughly speaking, the conjecture says that the algebraic $K$- and $L$-groups $K_n(\mathbb Z[G]), L_n(\mathbb Z[G]), n\in\mathbb Z$ of the integral group ring $\mathbb Z[G]$ of a group $G$ is determined 
by those of its virtually cyclic subgroups and the group homology of $G$. The conjecture was first formulated in \cite{fj} by Farrell and Jones with coefficients in $\mathbb Z$. Then in \cite{DL},
Davis and L\"uck gave a general framework for the formulations of various isomorphism conjectures in $K$-and $L$-theories. In these formulations, the coefficients are untwisted rings.  Later on, the conjecture was extended by Bartels and Reich \cite{BR} to allow for coefficients in any additive category $\mathcal A$ with a right $G$-action. The precise formulation of the conjecture will be given in Section \ref{section 2.1}. 

The significance of the Farrell-Jones conjecture not only lies in the fact that it provides 
a tool for the computations of algebraic $K$- and $L$-theories of groups rings, but also implies many 
important conjectures in geometry, topology and algebra. In particular, if both the $K$- and $L$-theoretic  
FJC with coefficients in $\mathbb Z$ hold for a group $G$, then the Borel conjecture holds for $G$ in dimensions greater than or equal to 5. If the $L$-theoretic FJC with coefficients in $\mathbb Z$ holds for $G$, then the Novikov conjecture on homotopy invariance of higher signatures holds for $G$. See \cite{LR} for a survey on the FJC and its applications. 

We call a group \textit{FJ} if it satisfies the $K$- and $L$-theoretic FJC with coefficients in $\mathbb Z$.  We denote the class of all FJ groups by $\mathcal {FJ}$.  In general, in application of the FJC to the Borel conjecture, one needs to show $G\in\mathcal{FJ}$ if one wants to prove the Borel conjecture for $G$ in dimensions $\ge 5$. Interestingly,  one of our main result shows the following:

%that in order for the simple Borel conjecture to hold for an arbitrary semi-direct product $G\rtimes\mathbb Z$, it suffices if $G\in\mathcal {FJ}$ and in order for the bordism Borel conjecture to hold for an arbitrary semi-direct product $G\rtimes\mathbb Z$, it suffices if $G\in\mathcal{FJ}$ and $Wh(G\times\mathbb Z)=0$, which is always true for all known torsion free examples of $G\in\mathcal{FJ}$.

\begin{thmA}\label{A} If $G\in\mathcal{FJ}$, then the simple Borel conjecture holds for every semi-direct product $G\rtimes\mathbb Z$ in dimensions  $\ge 5$. If in addition $Wh(G\times\mathbb Z)=0$, then the bordism Borel conjecture holds for every semi-direct product $G\rtimes\mathbb Z$ in dimensions $\ge 5$. 
\end{thmA}

An immediate corollary of Theorem A is the following:

\begin{cor}\label{cor1.1} Let $M$ be a closed aspherical manifold which fibers over the unit circle with fiber $N$.  If $\pi_1(N)\in\mathcal{FJ}$ and dim$(M)\ge 5$, then the simple Borel conjecture holds for $M$. If in addition $Wh(\pi_1(N)\times\mathbb Z)=0$, then the bordism Borel conjecture holds for $M$. 
\end{cor}

Following \cite{BL2}, we define a class of groups $\mathcal B$ in the following way, except we add more groups to the list.

\begin{defn}\label{C} Let $\mathcal B$ be the smallest class of groups so that:
\begin{enumerate}
\item $\mathcal B$ contains all  groups from the following classes of groups: CAT(0) groups, Gromov hyperbolic groups,  lattices in virtually connected Lie groups,  virtually solvable groups, fundamental groups of graphs of abelian groups, $S$-arithmetic groups.
\item  $\mathcal B$ is closed under taking subgroups, finite direct products of groups, free products of groups, direct colimits (with not necessary injective structure maps) of groups.
\item For any group homomorphism $\phi: G\rightarrow H$. If  $H\in\mathcal B$ and $\phi^{-1}(V)\in\mathcal B$ for every virtually cyclic subgroup $V\subseteq H$, then $G\in\mathcal B$.
\end{enumerate}
\end{defn}

\begin{cor}\label{cor1.2} If $G\in\mathcal B$,  then the simple Borel conjecture and the bordism Borel conjecture holds for every semi-direct product $G\rtimes\mathbb Z$ in dimensions $\ge 5$. 

\end{cor}

\begin{proof}  Note that we can assume $G$ is torsion free, otherwise $G$ is not the fundamental group of any closed aspherical manifold. Now we have $\mathcal B\subseteq\mathcal {FJ}$ since the groups in part (1) of Definition \ref{C} and the groups obtained from these groups by the operations in parts (2) and (3)  satisfy the $K$- and $L$-theoretic FJC with coefficients in $\mathbb Z$ (with coefficients in any additive category indeed), see \cite{BL2},\cite{BLR2},\cite{WC},\cite{WC1},\cite{BFL},\cite{KLR},\cite{GMR},\cite{FW1},\cite{FW2},\cite{RH}. Therefore, Corollary \ref{cor1.2} follows from Theorem A and the fact that $Wh(G\times\mathbb Z)=0$ if $G$ is torsion free and $G\in\mathcal B$. This is because $G\times\mathbb Z\in\mathcal B\subseteq\mathcal{FJ}$ and Whitehead groups of torsion free FJ groups vanish.
\end{proof}

One of the main ingredients in proving Theorem A is the following result, which is of course of independent interest:

\begin{thmB} Let $G$ be a torsion free group and $G\rtimes\mathbb Z$ be any semi-direct product of $G$
with $\mathbb Z$.
\begin{enumerate}
\item Suppose the $L$-theoretic FJC with coefficients in $\mathbb Z$ holds for  $G$, then it also holds for  $G\rtimes\mathbb Z$.
\item Suppose the $K$-theoretic FJC with coefficients in $\mathbb Z$ holds for  $G$, then there are obstruction groups, $Nil_n^{G\rtimes\mathbb Z}, n\in\mathbb Z$, associated to the given data, so that $G\rtimes\mathbb Z$ satisfies the $K$-theoretic FJC with coefficients in $\mathbb Z$ if and only if $Nil_n^{G\rtimes\mathbb Z}=0, \forall n\in\mathbb Z$.
\end{enumerate}
\end{thmB}

\begin{rem} As one sees,  we are not able to prove the same result as in part (1) of Theorem B for the $K$-theoretic FJC with coefficients in $\mathbb Z$. Therefore, some additional argument is needed in order to deduce Theorem A from Theorem B. This is achieved by proving some general result  on $L$-groups. See Lemma \ref{l group}.
\end{rem}

\begin{rem}\label{general} Our result is indeed more general. It still holds  if we replace the coefficient ring $\mathbb Z$
by any right $G$-additive category $\mathcal A$ with involution in part (1) and replace the coefficient ring $\mathbb Z$ by any associative ring with unit which is \textit{regular} in part (2) (a  ring is called \textit{regular} if it is left Noetherian and every finitely generated left $R$-module has a projective resolution of finite type). For example, let $\mathcal A$
be a right $G\rtimes\mathbb Z$-additive category with involution, view $\mathcal A$ as a right $G$-additive category with involution in the natural way.  Then  if the $L$-theoretic FJC with coefficient in $\mathcal A$ holds for $G$,  the $L$-theoretic FJC with coefficient in $\mathcal A$  also holds for $G\rtimes\mathbb Z$.
\end{rem}

\begin{rem} A key step in the proof of FJC for some classes of groups, including the class of Baumslag-Solitar groups
\cite{FW1},\cite{FW2} and the class of solvable groups \cite{WC1}, is to prove the conjecture for a group of the form $G\rtimes_\alpha\mathbb Z$, where $G$ is a torsion free abelian group and $\alpha$ is some special action of $\mathbb Z$ on $G$. Complicated geometric arguments  are used in these works in order to prove the conjecture for $G\rtimes_\alpha\mathbb Z$.  Therefore, a general result as in part (1) of Theorem B is very useful.
\end{rem}

%As an immediate consequence of part (1) of Theorem B, we have:
%
%\begin{cor}\label{novikov} If the $L$-theoretic FJC with coefficient in $\mathbb Z$ holds for a torsion free group $G$, then the Novikov conjecture holds for any repeated semi-direct product  $\big(((G\rtimes\mathbb Z)\rtimes\mathbb Z)\cdots\big)\rtimes\mathbb Z$. 
%\end{cor}
%
%\begin{rem} Corollary \ref{novikov} enables us to reprove the Novikov conjecture very easily for the class  of finitely generated free abelian groups (for which the conjecture was firstly proved): we can just apply Corollary \ref{novikov} to the case when $G$ is the trivial group. 
%\end{rem}

In the literature, the term FJC may refer to different versions of the conjecture and they have different inheritance properties.  For convenience, we introduce the following. Let $G$ be a group and $\mathcal A$ be a right $G$-additive category. We say a group $G$  satisfies FJC$_{\mathcal A}$ if $G$ satisfies FJC with coefficient in $\mathcal A$. We say $G$ satisfies FJC$_{Ad}$ if it satisfies FJC$_{\mathcal A}$ for every additive category $\mathcal A$ with a right $G$-action. When $\mathcal A$ is the category of finitely generated left free $R$-modules for some associative ring $R$ with trivial group action, we then denote FJC$_{\mathcal A}$ by FJC$_{R}$.

The version FJC$_{Ad}$ has some nice inheritance properties.  For example, FJC$_{Ad}$ is closed under 
the operations in parts (2) and (3) of Definition \ref{C}, see \cite[Section 2.3]{BFL} for a summary.
However, these inheritance properties are unknown for  FJC$_{\mathcal A}$; except it was proved in
\cite[Theorem 0.8]{BEL} that FJC$_{\mathcal A}$ is closed under taking direct colimit of groups with injective structure maps. Theorem B is probably the second inheritance property for the $L$-theoretic FJC$_{\mathcal A}$.

Theorem B together with some nice inheritance properties of FJC$_{Ad}$  implies the following:

\begin{cor}\label{extension2} Let $1\longrightarrow K\longrightarrow G\longrightarrow Q\longrightarrow 1$ be an extension of groups. Suppose both $K$ and $Q$ are torsion free and satisfy the $L$-theoretic FJC$_{Ad}$, then $G$ also satisfies the $L$-theoretic FJC$_{Ad}$.
\end{cor}
\begin{proof} It is well-known that an infinite virtually cyclic group either admits a surjection onto the infinite cyclic group with finite kernel or admits a surjection onto the infinite Dihedral group with finite kernel. Therefore a torsion free virtually cyclic group is either trivial or infinite cyclic.   Hence, by the fact that FJC$_{Ad}$ is closed under operation (3) in Definition \ref{C}, see \cite[Theorem 2.7]{BFL}\cite[Theorem A]{HPR1}, it suffices to show that for any infinite cyclic subgroup of $Q$, its inverse image in $G$ satisfies the $L$-theoretic FJC$_{Ad}$. Note that its inverse image in $G$ is isomorphic to a group of the form $K\rtimes_{\alpha}\mathbb Z$.  Let $\mathcal A$ be any additive $K\rtimes_\alpha\mathbb Z$ category. By assumption, $K$ is torsion free and satisfies the $L$-theoretic  FJC$_{\mathcal A}$, so by Theorem B and Remark \ref{general}, $K\rtimes_\alpha\mathbb Z$ also satisfies the $L$-theoretic FJC$_{\mathcal A}$. Since $\mathcal A$ can be arbitrary, we see that $K\rtimes_{\alpha}\mathbb Z$ satisfies FJC$_{Ad}$. This completes the proof.
\end{proof}

Let $\mathcal C$ be a class of groups. A group $G$ is called \textit{poly-$\mathcal C$} if it has a filtration by subgroups $$1=G_0\subseteq G_1\subseteq G_2\subseteq\cdots\subseteq G_n=G$$
so that each $G_{i-1}$ is normal in $G_i, i=1, 2, \cdots, n$ and each $G_i/G_{i-1}$ is in $\mathcal C$.

\begin{defn}\label{poly} Let $\mathcal B$ be as in Definition \ref{C}.  Define $\mathcal D\subseteq\mathcal B$ to be the  class of groups in $\mathcal B$ which are torsion free. 
\end{defn}

\begin{cor} The class of poly-$\mathcal D$-groups satisfies the $L$-theoretic FJC$_{Ad}$. Therefore, they also satisfy the Novikov conjecture.
\end{cor}
\begin{proof}   By repeatedly applying Corollary \ref{extension2}, we see that the $L$-theoretic FJC$_{Ad}$ holds for poly-$\mathcal D$ groups. 
\end{proof}

\begin{rem} The best result so far for the Novikov conjecture is due to Guoliang Yu, who proved in \cite{YGL} the coarse Baum-Connes conjecture for groups which admit a uniform
   embedding into Hilbert space. This is a huge class of groups. However, there are groups, constructed as direct colimits of Gromov hyperbolic groups, which do not admit a uniform embedding into Hilbert space, see for example \cite{HLS}. Note that these groups lie in the class $\mathcal B$. Also it is still an open question that whether CAT(0)-groups and their extensions admit a uniform embedding into Hilbert space. Therefore,  poly-$\mathcal D$-groups contain new examples of groups that satisfy the Novikov conjecture. 

\end{rem}

The paper is organized as follows. In Section \ref{conalg}, we briefly recall the formulation of the Farrell-Jones conjecture and the controlled algebra approach to the conjecture. We use this approach to prove Theorem B. We also prove some lemmas on the equivariantly continuously controlled condition. These lemmas are important in proving the main theorem (Theorem \ref{equivalence}) of Section \ref{equiob}, which is a key step in proving Theorem B. We prove our main theorems in the final section, in which we first prove an important lemma (Lemma \ref{l group}) for $L$-groups.  This lemma together with Theorem B enable us to prove Theorem A. 

\vskip 10pt
\noindent
\textbf{Acknowledgement.}
The author would like to thank Professor Bruce Hughes for his support and  helpful comments about the paper. The author would also like to thank his thesis advisor Professor Jean-Fran\c{c}ois Lafont for his enormous and continuous support throughout the years. Part of the project was initiated when the author was a graduate student at the Ohio State University under the guidance of Jean. The author also thanks Professor Guoliang Yu for his support and encouragement. The author learned a lot through numerous discussions with Guoliang. The author also thanks Shanghai Center for Mathematical Science for its hospitality during the author's visit in Summer 2015.   Part of the project was done during this visit.

%%%%%%%%%%%%%%%%%%%%%%%%%%%%%%%%%%%%%%%%%
%%%%%%%%%%%%%%%%%%%%%%%%%%%%%%%%%%%%%%%%%%

%%%%%%%%%%%%%%%%%%%%%%%%%%%%%%%%
%%%%%%%%%%%%%%%%%%%%%%%%%%%%%%%%
%%%%%%%%%%%%%%%%%%%%%%%%%%%%%%%%
%%%%%%%%%%%%%%%%%%%%%%%%%%%%%%%%
%%%%%%%%%%%%%%%%%%%%%%%%%%%%%%%%

\section{The Farrell-Jones conjecture and controlled algebra}\label{conalg}

In the first half of the section, we briefly recall the formulation of the FJC  and the  controlled algebra approach to the conjecture. More details can be found in \cite{fj}, \cite{DL}, \cite{BFJR}, \cite{BR}. This approach will  be important to our treatment later. In the second half of the section, we prove some lemmas on the \textit{equivariantly continuously controlled} condition (see Definition \ref{ecc}) that will be important in later sections. 

\subsection{The Farrell-Jones conjecture}\label{section 2.1} Let $G$ be a group and $\mathcal A$ be an additive category with a right $G$-action $\alpha$ (we always assume  $\mathcal A$ comes with an involution, which is compatible with the right $G$-action, when we talk about $L$-theory).   One can  form the  ``twisted group additive category" $\mathcal A_\alpha[G]$ \cite[Definition 2.1]{BR} (this category is denoted by $\mathcal A*_Gpt$ in \cite{BR}, but it is more enlightening to denote this category by $\mathcal A_\alpha[G]$ for our purpose). Its definition is recalled in Definition \ref{groupcat} and it will be important to our treatment later.  When 
$\mathcal A$ is the additive category of finitely generated left free $R$-modules with the trivial $G$-action, $\mathcal A[G]$ is equivalent to the additive category of finitely generated left  free $R[G]$-modules. 
Now  let $H^G_*(-; \textbf{K}_{\mathcal A})$ and $H^G_*(-;\textbf{L}^{<-\infty>}_{\mathcal A})$ be the two $G$-equivariant homology theories constructed by Bartels and Reich in \cite{BR}, using the method of Davis and L\"uck \cite{DL}. These two equivariant homology theories have the property that for any subgroup $H<G$, $H^G_n(G/H; \textbf{K}_{\mathcal A})=K_n(\mathcal A_{\alpha}[H]),  H^G_n(G/H;\textbf{L}^{<-\infty>}_{\mathcal A})=L_n^{<-\infty>}(\mathcal A_\alpha[H]), \forall n\in\mathbb Z$, where the decoration $-\infty$ on the $L$-groups means we are dealing with Ranicki's ultimate $L$-groups. In particular, $H^G_n(pt; \textbf{K}_{\mathcal A})=K_n(\mathcal A_{\alpha}[G])$ and $H^G_n(pt;\textbf{L}^{<-\infty>}_{\mathcal A})=L_n^{<-\infty>}(\mathcal A_\alpha[G])$. 

We also need the notion of \textit{classifying space} of a group relative to a family of its subgroups. For $\mathcal F$,  a family of subgroups of $G$ which is closed under taking subgroups and conjugations, denote by $E_{\mathcal F}G$ a model for the classifying space of $G$ relative to the family $\mathcal F$. It is a $G$-CW complex and is characterized, up to
$G$-equivariant homotopy equivalence, by the properties that every isotropy group of the action lies in the family $\mathcal F$, and the fixed point set of $H$ is contractible if $H\in\mathcal F$ and empty if $H\notin\mathcal F$. 
Examples of classifying spaces are $EG, E_{\mathcal {FIN}}G$, and $E_{\mathcal VC}G$, corresponding to the families of trivial, finite, and virtually cyclic  subgroups  of $G$ respectively.  

\vskip 10pt
\noindent
\textit{The Farrell-Jones Conjecture.}\label{FJC} Let $G$ be a group and $\mathcal A$ be an additive category with a right $G$-action. We say the group $G$ satisfies the $K$-theoretic FJC with coefficient in $\mathcal A$ if the $K$-theoretic assembly map 
\begin{align}\label{k assembly}
& A{_{K}^{\mathcal A}}: H^G_n(E_\mathcal{VC}G; \textbf{K}_{\mathcal A})\rightarrow H^G_n(pt; \textbf K_{\mathcal A})=K_n(\mathcal A_\alpha[G])
\end{align}
which is induced by the obvious map $E_\mathcal{VC}G\rightarrow pt$, is an isomorphism for all $n\in\mathbb Z$. We say the group $G$ satisfies the $L$-theoretic FJC with coefficient in $\mathcal A$ if  the 
corresponding $L$-theoretic assembly map
\begin{align}\label{l assembly}
& A_L^{\mathcal A}: H^G_n(E_\mathcal{VC}G; \textbf{L}^{<-\infty>}_{\mathcal A})\rightarrow H^G_n(pt; \textbf L^{<-\infty>}_{\mathcal A})=L^{<-\infty>}_n(\mathcal A_\alpha[G])
\end{align}
is an isomorphism for all $n\in\mathbb Z$.

\subsection{Controlled algebra approach to the FJC} Let $\mathcal A$ be a (small) additive category with a right $G$-action and $X$ be a left $G$-space. The additive category $\mathcal C(X; \mathcal A)$ of \textit{geometric modules} over $X$ with coefficient in $\mathcal A$ is defined as follows: objects are functions $A: X\rightarrow Ob(\mathcal A)$ with \textit{locally finite support,} i.e. $\text{supp}A=\{x\in X|\ A_x\ne 0\}$ is a locally finite subset of $X$, meaning every point in $X$ has an open neighborhood whose intersection with supp$A$ is finite. An object will usually be denoted by $A=(A_x)_{x\in X}$. A morphism $\phi: A\rightarrow B$ is a matrix of morphisms $(\phi_{y,x}: A_x\rightarrow B_y)_{(x,y)\in X\times X}$ such that there are only finitely many nonzero entries in each row and each column. Compositions of morphisms are given by matrix multiplications. More precisely
$$(\psi_{z, y}: B_y\rightarrow C_z)\circ(\phi_{y, x}: A_x\rightarrow B_y)=(\sum_{y\in X}\psi_{z, y}\circ\phi_{y, x}: A_x\rightarrow C_z)$$
Note that the sum on the right hand side is finite.
There is a right $G$-action on $\mathcal C(X;\mathcal A)$, which is given by 
$$(g^{*}A)_x= g^*(A_{gx}), \  (g^*\phi)_{y,x}=g^{*}(\phi_{gy, gx})$$ and the fixed category is denoted by $\mathcal {C}^{G}(X;\mathcal A)$.

For any object $A$ and morphism $\phi$ in $\mathcal C(X; \mathcal A)$, their supports are the following sets:
$$\text{supp}A=\{x\in X|\ A_x\ne 0\}, \ \ \text{supp}\ \phi=\{(x,y)\in X\times X|\ \phi_{y,x}\ne 0\}$$
In order to get interesting subcategories of $\mathcal C(X; \mathcal A)$, one can prescribe certain support conditions on objects and morphisms. The convenient language for this purpose is the notion of  \textit{coarse structures} on spaces introduced in \cite{HPR}. We recall the definition here (with a slight modification for our purpose).

\vskip 5 pt
\begin{defn}\label{coarse} A\textit{ coarse structure} ($\mathcal E, \mathcal F$) on a set $X$ is a collection $\mathcal E$ of subsets of $X\times X$, and a collection $\mathcal F$ of subsets of $X$  satisfying the following properties: 

(1) If $E', E''\in\mathcal E$, then $E'\cup E''\subseteq E$ for some $E\in \mathcal E$;

(2) If $E', E''\in\mathcal E$, then $E'\circ E''=\{(x,y)\in X\times X|\ \exists z\in X\ \text{s.t.}(x,z)\in E'\ \text{and}\ (z,y)\in E'' \}\subseteq E$ for some $E\in\mathcal E$;

(3) The diagonal $\vartriangle=\{(x,x)|\ x\in X\}$ is contained in some $E\in\mathcal E$;

(4) If $F', F''\in\mathcal F$, then $F'\cup F''\subseteq F$ for some $F\in\mathcal F$.
\end{defn}

If there is a $G$-action on the set $X$, we then require every member in $\mathcal E$ and $\mathcal F$ to be $G$-invariant, where $G$ acts on $X\times X$ diagonally.  If $p: Y\rightarrow X$ is a $G$-equivariant map, then the pullback ($(p\times p)^{-1}\mathcal E, p^{-1}\mathcal F)$ is a coarse structure on $Y$.

Now if $(\mathcal E, \mathcal F)$ is a coarse structure on a $G$-space $X$, one can define a subcategory $\mathcal C(X, \mathcal E,\mathcal F;\mathcal A)$ of $\mathcal C(X; \mathcal A)$, with object and morphism supports contained in members of $\mathcal F$ and $\mathcal E$ respectively (in addition to the general  finiteness conditions on them).  $G$ acts on  this additive subcategory and the fixed subcategory is denoted by $\mathcal C^G(X, \mathcal E, \mathcal F; \mathcal A)$. The pair $(\mathcal E, \mathcal F)$ are  usually referred to as control conditions on morphisms and objects.

One of the control condition on morphisms, which is used to construct a model for the assembly maps in the  FJC,   is the \textit{equivariantly continuously controlled} condition introduced in \cite{BFJR}.  It is a generalization  to the equivariant setting of the \textit{continuously controlled} condition introduced in \cite{ACFP}.   We recall the definition here.

\begin{defn}\label{ecc}(\cite[Definition 2.7]{BFJR}) Let $X$ be a topological space with a left $G$-action by homeomorphisms. Equip $X\times[1, \infty]$ with the diagonal $G$-action, where $G$-acts trivially on $[1, \infty]$. A subset $E\subseteq (X\times [1, \infty))^2 $ is called \textit{equivariantly continuously controlled} if the following holds:

(i) For every $x\in X$ and  $G_x$-invariant open neighborhood  $U$ of $(x, \infty)$ in $X\times[1, \infty]$, there exists a $G_x$-invariant neighborhood $V\subseteq U$ of $(x, \infty)$ in $X\times[1, \infty]$ such that 
$$(U^c\times V)\cap E=\emptyset$$
where $U^c$ denotes the complement of $U$ in $X\times [1, \infty]$.

(ii) There exists $\alpha>0$, depending on $E$,  such that if $(x,s)\times(x',s')\in E$, then $|s-s'|<\alpha$;

(iii) $E$ is symmetric, i.e. if $(p,q)\in E$, then $(q,p)\in E$;

(iv) $E$ is invariant under the diagonal action of $G$.
\end{defn}

The collection of $G$-equivariantly continuously controlled subsets of $(X\times [1,\infty))^2$ will be denoted by $\mathcal E^X_{Gcc}$.  It satisfies the conditions (1)-(3) in Definition \ref{coarse}.

\begin{defn}\label{obstruction}(\cite[Section 3.2]{BFJR}\cite[Section 3.3]{BLR2}) For any left $G$-space $X$ and additive category $\mathcal A$ with a right $G$-action, one defines the following categories:

(1) $\mathcal O^G(X; \mathcal A)=\mathcal C^G(G\times X\times [1,\infty), (p\times p)^{-1}\mathcal E^X_{Gcc}\cap (r\times r)^{-1}\mathcal E_G, q^{-1}\mathcal F_{Gc}; \mathcal A)$, where $p: G\times X\times [1,\infty)\rightarrow X\times [1,\infty)$, $q: G\times X\times [1,\infty)\rightarrow G\times X$ and $r: G\times X\times[1,\infty)\rightarrow G$ are projections, $\mathcal E_G=\{E\subseteq G\times G\ |\ gE=E\ \text{for\ all}\ 
g\in G\ \text{and}\ \exists\ \text{finite}\ S\subseteq G\ \text{s.t.}\ \text{for all}\ (g, g')\in E,\ \text{we have}\ g^{-1}g'\in S \}$ and $\mathcal F_{Gc}$ consists of $G$-cocompact subsets of $G\times X$, i.e. subsets of the form $G\cdot K\subseteq G\times X$, where $K\subseteq G\times X$ is compact;

(2) $\mathcal T^G(X; \mathcal A)$ is the full subcategory of $\mathcal O^G(X; \mathcal A)$ consisting of those objects $A$ with the following property: there exists $C>1$ such that if $A_{(g,x,t)}\ne 0$, then $t<C$;

(3) $\mathcal D^G(X; \mathcal A)$ is the quotient category of $\mathcal O^G(X; \mathcal A)$ by the full subcategory $\mathcal T^G(X; \mathcal A)$: it has the same objects as $\mathcal O^G(X; \mathcal A)$, and any morphism from $A$ to $B$ in $\mathcal D^G(X; \mathcal A)$ is
represented by a morphism $\phi: A\rightarrow B$ in $\mathcal O^G(X; \mathcal A)$, with two morphisms $\phi, \psi: A\rightarrow B$  identified if their difference $\phi-\psi$ factors through an object in $\mathcal T^G(X; \mathcal A)$.

\end{defn}

\begin{rem} In \cite{BFJR}, the morphism control condition $\mathcal E_G$ was not required in the definitions. It was added into the definitions in \cite{BLR2}\cite{BL2} because it was important for proving
the conjecture for  hyperbolic and CAT(0)-groups. Although this addition changes the categories, it does not change the theory. Later on, we will see that this is also important for our treatment. 
\end{rem}

These constructions define functors from the category of $G$-CW complexes to the category of additive categories. The importance of these categories lies in the following facts:

\begin{thm}\label{obstruction1} The following results hold:

(i) The sequence $$\mathcal T^{G}(X; \mathcal A)\rightarrow \mathcal O^{G}(X; \mathcal A)\rightarrow\mathcal D^{G}(X; \mathcal A)$$ is a Karoubi filtration, hence gives rise to a fibration sequence $$\mathbb K^{-\infty}(\mathcal T^{G}(X; \mathcal A))\rightarrow\mathbb{K}^{-\infty}(\mathcal O^G(X; \mathcal A))\rightarrow\mathbb K^{-\infty}(\mathcal D^{G}(X; \mathcal A))$$ of spectra after applying the non-connective $K$-theory, and therefore a long exact sequence on $K$-groups;

(ii) The additive category $\mathcal O^G(pt; \mathcal A)$ has trival $K$-groups;

(iii) The functor $\pi_*: \mathcal T^G(X; \mathcal A)\rightarrow \mathcal T^G(pt; \mathcal A)$ induced by the obvious map $\pi: X\rightarrow \{pt\}$ is an equivalence of categories, hence induces isomorphisms on $K$-groups; 

(iv) There is a natural isomorphism between the two functors $H^G_*(-; \textbf K_\mathcal A)$ and $K_{*+1}(\mathcal D^G(-; \mathcal A)):=\pi_{*+1}(\mathbb K^{-\infty}(\mathcal D^G(-; \mathcal A)))$ from the category of $G$-CW complexes to the category of graded abelian groups.  In particular, the map
\begin{align}
K_{*+1}(\mathcal D^G(E_{\mathcal {VC}}G; \mathcal A))\rightarrow K_{*+1}(\mathcal D^G(pt; \mathcal A))
\end{align}
is equivalent to the assembly map \ref{k assembly}. 
\end{thm}

For information about Karoubi filtrations, see \cite{CP}. Fact (ii) can be proved by an Eilenberg swindle argument. Fact (iii) can be checked directly. Fact (iv) is first proven in \cite{BFJR} for coefficients in rings, and proven for  coefficients in additive categories  in \cite{BR}. One has the same constructions and results for $L$-theory and details can be found in \cite{BL2}. The above facts implies the following \cite{BLR2}\cite{BL2}:

\begin{thm}\label{vanishing} The $K$-theoretic FJC$_{\mathcal A}$ holds for $G$  if and only if the $K$-theory of  $\mathcal O^{G}(E_{\mathcal{VC}}G; \mathcal A)$ is trivial, i.e.  $K_n(\mathcal O^{G}(E_{\mathcal{VC}}G; \mathcal A))=0, \forall n\in\mathbb Z$. The corresponding statement is true for the $L$-theoretic FJC$_{\mathcal A}$.
\end{thm}

Because of this theorem, the category $\mathcal O^G(E_{\mathcal {VC}}G; \mathcal A)$ is usually referred to as the \textit{obstruction category}.

\subsection{Some lemmas on equivariant continuous control} In this subsection, we prove some general results about the equivariantly  continuously controlled condition. These results will be important in proving Theorem \ref{equivalence}, which is one of the key ingredients in proving Theorem B. The proofs of some of these results and the proof of Theorem \ref{equivalence} in the next section are quite technical and readers may want to jump to Section \ref{main} to see the proof of the main theorems first  (assuming Theorem \ref{equivalence}.)

Now let $X$ be a $G$-CW complex and $H<G$ be a subgroup. Then $X$ is also an $H$-CW complex in a natural way. Thus we can consider $\mathcal E^X_{Gcc}$ and $\mathcal E^X_{Hcc}$. In the following several lemmas we study the relation between these two control conditions. 

\begin{lem}\label{11} $\mathcal E^X_{Gcc}\subseteq\mathcal E^X_{Hcc}$ if one of the following holds:

(i) The action of $G$ on $X$ is free.

(ii) $H<G$ is of finite index.
\end{lem}
\begin{proof} (i) This is obvious.\\ 
(ii) The only nontrivial part is (i) of Definition \ref{ecc}. But this can be easily shown, by noting $\bigcap_{g\in G_x}gU$ is a $G_x$-invariant neighborhood of $(x, \infty)$ in $X\times [1, \infty]$ for any $H_x$-invariant neighborhood $U$ of 
$(x, \infty)$ in $X\times [1, \infty]$,  since $[G_x: H_x]\le[G:H]<\infty$.
\end{proof}

\begin{lem}\label{normal} If $H$ is a normal subgroup of $G$, then for any $E\in\mathcal E^X_{Hcc}$ and $g\in G$, we have $gE\in\mathcal E^X_{Hcc}$. 
\end{lem}
\begin{proof} Since $H$ is a normal subgroup of $G$, one sees easily that $H\cdot gE=gE$, hence $gE$ is $H$-invariant. One also easily sees $gE$ is symmetric, bounded in the $[1, \infty)$ direction. Now for any 
$x\in X$, any $H_x$-invariant neighborhood $U$ of $(x, \infty)$ in $X\times [1, \infty]$, we consider $g^{-1}x$. Since $H<G$ is normal, we have $H_{g^{-1}x}=g^{-1}H_xg$. Hence $g^{-1}U$ is an $H_{g^{-1}x}$-invariant neighborhood of 
$(g^{-1}x, \infty)$. Therefore there exists an $H_{g^{-1}x}$-invariant neighborhood $V'\subseteq g^{-1}U$  
of $(g^{-1}x, \infty)$ so that 
$$((g^{-1}U)^c\times V')\cap E=\emptyset$$
Hence
$$(U^c\times gV')\cap gE=\emptyset$$
This completes the proof by letting $V=gV'$ and by noting $V$ is $H_x$-invariant. 
\end{proof}

The  proof of part (iv) in the proof of the following lemma is motivated from the proof of \cite[Lemma 3.3]{BFJR}.

\begin{lem}\label{control} Let $E\in\mathcal E^X_{Hcc}$. For any compact subset $K\subseteq X$, define
$$E'=\{(x, s)\times(x', s')|\ \exists g, g'\in G\ \text{s.t.}\ g^{-1}g'\in H, g^{-1}x, g'^{-1}x'\in K, (g^{-1}x, s)\times(g^{-1}x', s')\in E \}$$
Then $E'\in\mathcal E^X_{Gcc}$.
\end{lem}
\begin{proof} (i) $E'$ is symmetric: suppose $(x,s)\times(x',s')\in E'$, then we have $g, g'\in G$ with the properties in the definition of $E'$.  Now since $E$ is symmetric,  we have $(g^{-1}x', s')\times(g^{-1}x, s)\in E$.  Since $E$ is $H$-invariant and $g^{-1}g'\in H$, we get $(g'^{-1}x', s')\times (g'^{-1}x, s)\in E$. This shows $g', g$ fullfill the  properties for the pair $(x',s')\times(x,s)$ in the definition of $E'$. Hence $E'$ is symmetric;

(ii) $E'$ is $G$-invariant: suppose $(x,s)\times(x',s')\in E'$ and $l\in G$. Let $g, g'$ be as before. One easily verifies that $lg, lg'$ fulfill the properties in the definition of $E'$ for $(lx, s)\times (lx', s')$.  Hence $(lx, s)\times (lx', s')\in E'$, which shows $E'$ is $G$-invariant;

(iii) Bounded control in the $s$-direction: this is clear since $E$ is controlled;

(iv) $E'$ is $G$-equivariantly continuously controlled: suppose not, then there exists $x_0\in X$,  a $G_{x_0}$-invariant open neighborhood $U$ of $x_0\in X$ and $r>0$, such that for all $G_{x_0}$-invariant open neighborhood $V\subseteq U$ of $x_0$ in $X$ and all $l>r$, we have
$$\big(U\times (r, \infty]\big)^c\times\big(V\times(l, \infty]\big)\cap E'\ne \emptyset$$

Now since $X$ is a $G$-CW complex, the \textit{Slice theorem} (see \cite[Proposition 3.4]{BFJR}) applies, so that we can find a descending sequence $\{V^k\}_{k\in\mathbb N}$ of open neighborhoods of $x_0$ in $X$ with the following property: 

(a) Each $V^k$  is $G_{x_0}$-invariant;

(b) $gV^k\cap V^k=\emptyset$ if $g\notin G_{x_0}$;

(c) $\bigcap_{k\ge 1}\overline{G\cdot V^k}=G\cdot x_{0}$.

\vskip 10pt
We may assume $V^k\subseteq U, \forall k\in\bbN$. Hence  
\begin{align}\label{1}
[(U\times (r, \infty])^c\times (V^k\times (k+r, \infty])]\cap E'\ne\emptyset
\end{align}
thus we can find a sequence
\begin{align}\label{2}
 (x_k, s_k)\times(x_k', s_k')\in[(U\times (r, \infty])^c\times (V^k\times (k+r, \infty])]\cap E'
\end{align}
 By the definition of $E'$, there exist $g_k, g_k'$ such that 
\begin{align}\label{3}
g_k^{-1}g_k'\in H,\ g_k^{-1}x_k, g'{_k}^{-1}x_k'\in K,\ (g_k^{-1}x_k, s_k)\times(g_k^{-1}x_k', s_k')\in E
\end{align}
Since $E$ is $H$-invariant and $g_k^{-1}g_k'\in H$, the above also implies 
\begin{align}\label{4}
(g_k'^{-1}x_k, s_k)\times(g_k'^{-1}x_k', s_k')\in E
\end{align}
Now since $K$ is compact, by passing to a subsequence, we may assume $g'{_k}^{-1}x'_k\rightarrow y$.  Now since $x'_k\in V^k$, we have  for every $n\in\mathbb N$ and all $k>n$, $g'{_k}^{-1}x'{_k}\in G\cdot V^k\subseteq G\cdot V^n$, this implies $y\in\overline{G\cdot V^n}$ for all $n$, thus $y\in\bigcap_{n\ge 1}\overline{G\cdot V^n}=G\cdot x_0$. Hence there is $g\in G$ so that $y=gx_0$. Therefore $g'{_k}^{-1}x'_k\rightarrow y=gx_0$, so $g^{-1}g'{_k}^{-1}x'_k\rightarrow x_0\in V^1$. Thus when $k$ is large enough $g^{-1}g'{_k}^{-1}\in G_{x_0}$. Note that $s'_k\rightarrow \infty$,  hence by \ref{4} and the control condition in the $s$-direction, we have $s_k\rightarrow\infty$. Now by \ref{2}, when $k$ is large enough, $x_k\notin U$, hence $g^{-1}g'{_k}^{-1}x_k\notin U$ since $g^{-1}g'{_k}^{-1}\in G_{x_0}$ and $U$ is $G_{x_0}$-invariant. Thus when $k$ is large enough, $g'{_k}^{-1}x_k\notin gU$. 

Now consider $gU$, it is $G_{y}$-invariant since $y=gx_0$ and $U$ is $G_{x_0}$-invariant. In particular, $gU$ is  $H_y$-invariant. Now because $E$ is $H$-equivariantly continuously controlled, we can find an $H_y$-invariant open neighborhood $W\subseteq gU$ of $y=gx_0$ in $X$ and $N>r$ such that 
\begin{align}\label{5}
[(gU\times (r, \infty])^c\times (W\times (N, \infty])]\cap E=\emptyset
\end{align}
However, on the one hand, $(g_k'^{-1}x_k, s_k)\times(g_k'^{-1}x_k', s_k')\in E$ for all $k$ by \ref{4}, while on the other hand, when $k$ is large enough, we showed $g'{_k}^{-1}x_k\notin gU$, $g'{_k}^{-1}x'_k\rightarrow y\in W$ and $s_k, s_k'\rightarrow\infty$, so we also have $(g_k'^{-1}x_k, s_k)\times(g_k'^{-1}x_k', s_k')\in (gU\times (r, \infty])^c\times (W\times (N, \infty])$.  Thus when $k$ is large enough, we have
\begin{align}\label{6}
(g_k'^{-1}x_k, s_k)\times(g_k'^{-1}x_k', s_k')\in [(gU\times (r, \infty])^c\times (W\times (N, \infty])]\cap E
\end{align}
This contradicts to \ref{5} and we complete the proof.
\end{proof}

%%%%%%%%%%%%%%%%%%%%%%%%%%%%%%%%%%%%%
%%%%%%%%%%%%%%%%%%%%%%%%%%%%%%%%%%%%%%

\section{Equivalence of two categories}\label{equiob}
The major goal of this section is to prove Theorem \ref{equivalence}.  Let us firstly make some preparations. Let $\mathcal A$ be a right $G$-additive category. For any left $G$-set $X$, Bartels and Reich \cite{BR} defined a new additive category $\mathcal A*_GX$. The special case when $X=pt$ will be important
in our treatment and we denote it by $\mathcal A_{\alpha}[G]$, where $\alpha$ denotes the right $G$-action on $\mathcal A$.  We now recall its definition here. 

\begin{defn}\label{groupcat}(\cite[Definition 2.1]{BR}) Objects of the category $\mathcal A_{\alpha}[G]$ are the same as the objects of $\mathcal A$. A morphism $\phi: A\rightarrow B$ from $A$ to $B$ in $\mathcal A _\alpha[G]$
is a formal sum $\phi=\sum_{g\in G}\phi^g\cdot g$, where $\phi^g: A\rightarrow g^*B$ is a morphism in $\mathcal A$ and there are only finitely many $g\in G$ with $\phi^g\not=0$. Addition of morphisms is defined in
the obvious way.  Composition of morphisms is defined as follows: let $\phi=\sum_{k\in G}\phi^k\cdot k: A\rightarrow B$ be a morphism from $A$ to $B$ and $\psi=\sum_{h\in G}\psi^h\cdot h: B\rightarrow C$ be a morphism
from $B$ to $C$, their composition is given by 
$$\psi\circ\phi:=\sum_{g\in G}\big(\sum_{k,h\in G, g=hk}k^{*}(\psi^h)\circ \phi^k \big)\cdot g$$
\end{defn}

\vskip 10pt

Now let $F$ and $G$ be two groups.  Suppose there is a group homomorphism $\alpha: F\rightarrow Aut(G), f\mapsto\alpha_f$.  We then can form the associated semi-direct product $\Gamma=G\rtimes_{\alpha}F$. Every element  
$\r\in\Gamma$ can be uniquely written as $\r=gf$ for some $g\in G$ and $f\in F$.  We have $gfg'f'=g\alpha_{f}(g')ff'$. In particular $\alpha_f(g)=fgf^{-1}$.  Now let $X$ be a left $\Gamma$-CW complex and $\mathcal A$ be a right $\Gamma$-additive category (with or without involution). Since $F$ and $G$ are subgroups of $\Gamma$, they naturally inherit actions on $X$ and $\mathcal A$.  So we  can consider the obstruction category $\mathcal O^G(X; \mathcal A)$. Recall from Definition \ref{obstruction} that 
$$\mathcal O^\Gamma(X; \mathcal A)=\mathcal C^\Gamma(\Gamma\times X\times [1,\infty), (p\times p)^{-1}\mathcal E^X_{\Gamma cc}\cap (r\times r)^{-1}\mathcal E_\Gamma, q^{-1}\mathcal F_{\Gamma c}; \mathcal A)$$
$$\mathcal O^G(X; \mathcal A)=\mathcal C^G(G\times X\times [1,\infty), (p\times p)^{-1}\mathcal E^X_{Gcc}\cap (r\times r)^{-1}\mathcal E_G, q^{-1}\mathcal F_{Gc}; \mathcal A)$$

In what follows, we first show in Lemma \ref{auto} that the actions of $F$ on $G, X$ and $\mathcal A$ induce a right $F$-action  on the additive category $\mathcal O^G(X; \mathcal A)$ (which, by abuse of notation, will also be denoted by $\alpha$). We then can form the additive category $\mathcal O^G(X;\mathcal A)_{\alpha}[F]$.  We prove in Theorem \ref{equivalence}  that the two additive categories, $\mathcal O^{\Gamma}(X; \mathcal A)$ and $\mathcal O^G(X;\mathcal A)_{\alpha}[F]$, are equivalent under some assumptions. In particular, these assumptions are satisfied if $F$ is finite or if the action of $\Gamma$ on
$X$ is free. 

In the proofs of  Lemma \ref{auto} and Theorem \ref{equivalence}, we will check everything very carefully. One might want to jump to Remark \ref{inter} first for a simple interpretation 
of the somewhat complicated formulas in the lemma below. 

\begin{lem}\label{auto} Let $f\in F$. For any objects $A, B$  and morphism $\phi: A\rightarrow B$ in $\mathcal O^G(X; \mathcal A)$, the formulas 
$$(\alpha^f A)_{(g, x, s)}:=f^*(A_{(\alpha_f(g), fx, s)})$$
$$(\alpha^f\phi)_{(g', x', s'), (g, x, s)}:=f^*(\phi_{(\alpha_f(g'), fx', s'), (\alpha_f(g), fx, s)})$$
where $(g, x, s), (g', x', s')\in G\times X\times [1, \infty)$, define an additive functor $\alpha^f: \mathcal O^G(X; \mathcal A)\rightarrow\mathcal O^G(X;\mathcal A)$.  Moreover, the assignment $f\mapsto\alpha^f$ defines a right $F$-action on the  obstruction category $\mathcal O^G(X; \mathcal A)$.
\begin{proof} The major part is to check for each $f\in F$,  $\alpha^f: \mathcal O^G(X; \mathcal A)\rightarrow\mathcal O^G(X; \mathcal A)$ is a well-defined additive functor. As soon as this is done, it is easy to see $\alpha$ defines a right $F$-action on $\mathcal O^G(X; \mathcal A)$.  We will omit the $s$-component in places where it is not important for the proof.

\noindent
(i) $\alpha^f A$ is $G$-invariant: for any $l\in G, (g, x)\in G\times X$, we have
\begin{align*}
 \big(l^*(\alpha^f A)\big)_{(g, x)}
&=l^*\big((\alpha^f A)_{(lg, lx)}\big)\\
&=l^{*}\big(f^*(A_{(\alpha_f(lg), flx)})\big)\\
&=(fl)^{*}\big(A_{(\alpha_f(lg), flx)}\big)\\
&=\big(\alpha_f(l)f\big)^{*}\big(A_{(\alpha_f(lg), flx)}\big)\\
&=f^*\big(\alpha_f(l)^*(A_{(\alpha_f(lg), flx)})\big)\\
&=f^*\big((\alpha_f(l)^{*}A)_{(\alpha_f(g), \alpha_f(l^{-1})flx)}\big)\\
&=f^*\big(A_{(\alpha_f(g), \alpha_f(l^{-1})flx)}\big)\ \ \ \ \ \ \ \ \ \  \ \ \text{Since}\ A\ \text{is}\ G\text{-invariant}\\
&=f^{*}\big(A_{(\alpha_f(g), fx)}\big)\\
&=(\alpha^fA)_{(g, x)}
\end{align*}
Thus $l^*(\alpha^f A)=\alpha^f A$ for all $l\in G$. This proves $\alpha^f A$ is $G$-invariant for all $f\in F$. 
\vskip 10pt
\noindent
(ii) $\alpha^f \phi$ is $G$-invariant: for any $l\in G, (g, x), (g', x')\in G\times X$, we have
\begin{align*}
\big(l^*(\alpha^f\phi)\big)_{(g', x'), (g, x)}
&=l^*[(\alpha^f\phi)_{(lg', lx'), (lg, lx)}]\\
&=l^*[f^*(\phi_{(\alpha_f(lg'), flx'), (\alpha_f(lg), flx)})]\\
&=(fl)^*[\phi_{(\alpha_f(lg'), flx'), (\alpha_f(lg), flx)}]\\
&=(\alpha_f(l)f)^*[\phi_{(\alpha_f(lg'), flx'), (\alpha_f(lg), flx)}]\\
&=f^*[\alpha_f(l)^*(\phi_{(\alpha_f(lg'), flx'), (\alpha_f(lg), flx)})]\\
&=f^*[(\alpha_f(l)^*\phi)_{(\alpha_f(g'), \alpha_f(l^{-1})flx'), (\alpha_f(g), \alpha_f(l^{-1})flx)}]\\
&=f^*[\phi_{(\alpha_f(g'), fx'), (\alpha_f(g), fx)}]\\
&=(\alpha^f\phi)_{(g', x'), (g, x)}\\
\end{align*}
Thus $l^*(\alpha^f\phi)=\alpha^f\phi$ for all $l\in G$.  This proves $\alpha^f\phi$ is $G$-invariant for all $f\in F$.
\vskip 10pt
\noindent
(iii) $\alpha^f (id_A)=id_{\alpha^f A}$: this is easy.

\vskip 10pt
\noindent
(iv) $\alpha^f(\phi\circ \psi)=(\alpha^f \phi)\circ(\alpha^f \psi)$:  in the following identities, we are taking sum over $(l, y)\in G\times X$. We have
\begin{align*}[\alpha^f(\phi\circ \psi)]_{(g', x'), (g, x)}&=f^{*}[(\phi\circ \psi)_{(\alpha_f(g'), fx'), (\alpha_f(g), fx)}]\\
&=f^*[\phi_{(\alpha_f(g'), fx'), (l, y)}\circ \psi_{(l, y), (\alpha_f(g), fx)}]\\
&=f^*[\phi_{(\alpha_f(g'), fx'), (l, y)}]\circ f^*[\psi_{(l, y), (\alpha_f(g), fx)}]\\
&=(\alpha^f \phi)_{(g', x'), (\alpha_{f^{-1}}(l), f^{-1}y)}\circ (\alpha^f \psi)_{(\alpha_{f^{-1}}(l), f^{-1}y), (g, x)}\\
&=[(\alpha^f \phi)\circ(\alpha^f \psi)]_{(g',x'),(g,x)}
\end{align*}
The last identity holds because $(\alpha_{f^{-1}}(l), f^{-1}y)$ runs over $G\times X$ as $(l, y)$ runs over $G\times X$. We thus get $\alpha^f(\phi\circ\psi)=(\alpha^f \phi)\circ(\alpha^f\psi)$ 
\vskip 10pt
\noindent
(v) Additivity: this is easy.

\vskip 10pt
\noindent
(vi) Object support: we have to show $\{g^{-1}x | (g, x, s)\in \text{supp}(\alpha^f A)\}$ is contained in a compact subset of $X$.  Note that supp$(\alpha^f A)=\{(g,x,s) | (\alpha_f(g), fx, s)\in \text{supp}(A)\}$. By definition, there exists a compact subset $K\subseteq X$, so that $\{\alpha_f(g^{-1})fx=fg^{-1}x | (\alpha_f(g), fx, s)\in \text{supp}(A)\}\subseteq K$. This implies $\{g^{-1}x | (g, x, s)\in \text{supp}(\alpha^f A)\}=\{g^{-1}x | (\alpha_f(g), fx, s)\in \text{supp}(A)\}\subseteq f^{-1}\cdot K$, which completes the proof since $f^{-1}\cdot K\subseteq X$ is compact. 

\vskip 10pt
\noindent
(vii) Morphism support: by definition, there exists $E\in\mathcal E^X_{Gcc}$, so that the projection of supp$(\phi)$ in $(X\times [1, \infty))^2$ is contained in $E$.
Now by definition of $\alpha^f \phi$, the projection of supp$(\alpha^f \phi)$ in $(X\times [1, \infty))^2$ is contained in 
$$E'=\{(x, s)\times(x', s') | (fx, s)\times (fx', s')\in E\}=f^{-1}\cdot E$$
Now since $G<\Gamma$ is a normal subgroup, by Lemma \ref{normal}, we have $E'=f^{-1}\cdot E\in\mathcal E^X_{Gcc}$. Hence $\alpha^f \phi$ respects the support condition in the $X\times [1, \infty)$-direction. But clearly
it respects the support condition in the $G$-direction since $\alpha_f: G\rightarrow G$ is an automorphism. Therefore $\alpha^f \phi$ respects the morphism support condition.

\vskip 10pt
Now if $\mathcal A$ has an involution, then $\alpha^f$ commutes with the induced  involution on $\mathcal O^G(X; \mathcal A)$. 
This together with the above verification work shows that for each $f\in F$, $\alpha^f$ is a well-defined additive functor from $\mathcal O^G(X; \mathcal A)$ to itself. 

Finally $\alpha^e$ is the identity functor, where $e\in F$ is the trivial element and for $f, f'\in F$ , we have 
\begin{align*}
(\alpha^{ff'}A)_{(g,x)}&=(ff')^*[A_{(\alpha_{ff'}(g), ff'x)}]\\
&=(f')^*[f^*(A_{(\alpha_{ff'}(g), ff'x)})]\\
&=(f')^*[(\alpha^{f}A)_{(\alpha_{f'}(g)), f'x)}]\\
&=\big(\alpha^{f'}(\alpha^{f}A)\big)_{(g, x)}\\
\end{align*}
Therefore $\alpha^{ff'}=\alpha^{f'}\alpha^f$ on objects. Similarly this identity holds on morphisms. Hence $f\mapsto\alpha^f$ defines a right $F$-action on $\mathcal O^G(X;\mathcal A)$. We thus complete the proof of the lemma.
\end{proof}
\end{lem}

\begin{rem} We have used $\alpha$ to denote both the action of $F$ on $G$ and the action of $F$ on $\mathcal O^G(X; \mathcal A)$. Note however that the former action is from left while the later action is from right. It will be clear from the context which action $\alpha$ stands for. 
\end{rem}

\begin{rem}\label{inter}
A convenient way to interpret the defining formulas in the above Lemma is as follows: for any object $A$ in $\mathcal O^G(X; \mathcal A)$, it is completely determined by its values at points of the form $(e, x, s)\in G\times X\times [1,\infty)$, where $e\in G$ is the identity element. But  $(e, x, s)$ is a point in $\Gamma\times X\times [1, \infty)$, using the $\Gamma$-action on $\Gamma\times X\times [1, \infty)$, $A_{(e, x, s)}$ can be used to define an object $\tau A$ living over $\Gamma\times X\times [1, \infty)$, i.e. by defining $(\tau A)_{(\r, x, s)}:=(\r^{-1})^{*}(A_{(e, \r^{-1}x, s)})$ ($\tau A$ is actually an object in $\mathcal O^{\Gamma}(X; \mathcal A)$ that we will show shortly). Then $A$ is just the restriction of $\tau A$ to $G\times X\times [1,\infty)$. Note for any $f\in F$, the map $(\r, x, s)\mapsto (\r f, x, s)$ is a $\Gamma$-equivariant self-homeomorphism of $\Gamma\times X\times [1, \infty)$, it thus induces an automorphism of $\mathcal O^{\Gamma}(X; \mathcal A)$. Denote this automorphism by $\alpha^f$, then $(\alpha^f(\tau A))_{(\r, x, s)}=(\tau A)_{(\r f^{-1},  x, s)}=f^{*}((\tau A)_{(f\r f^{-1}, fx, s)})$, the second equality is due to the  $\Gamma$-invariance of $\tau A$.  This formula restricts to $G\times X\times [1, \infty)$ is the defining formula on objects in the above lemma. From this point of view, Lemma \ref{auto} is morally apparent. 
\end{rem}

\begin{rem}
The naive definition $(\alpha^f A)_{(g, x, s)}:=A_{(\alpha_f(g), x, s)}$ doesn't work since the map $(g, x, s)\mapsto (\alpha_f(g), x, s)$ is not $G$-equivariant.
\end{rem}

We can now form the category $\mathcal O^G(X; \mathcal A)_{\alpha}[F]$ and are going to show $\mathcal O^{\Gamma}(X;\mathcal A)$ and $\mathcal O^G(X; \mathcal A)_{\alpha}[F]$ are equivalent as additive categories (with involution) under some conditions. Forgetting about the control
conditions, this is then not very hard to see intuitively. The one-one correspondence between objects of these two categories has already been explained in Remark \ref{inter}. Let us focus on
morphisms. By $\Gamma$-invariance, every morphism $\phi: A\rightarrow B$ in $\mathcal O^{\Gamma}(X; \mathcal A)$ is uniquely determined by its values on $(e, x, s)\times (\r', x', s')$, i.e. by $\phi_{(\r', x', s'), (e, x, s)}: A_{(e, x, s)}\rightarrow B_{(\r', x', s')}$. They can be grouped into families $\{\phi_{(gf, x', s'), (e, x, s)} |\ g\in G\}, f\in F$. Due to the morphism control condition in the $\Gamma$-direction, see Definition \ref{obstruction}, there are only finitely many $f\in F$ whose corresponding family is non-trivial.  For each of such $f$, the family indexed by $f$ determines a morphism $\phi^f: A\rightarrow \alpha^fB$ in $\mathcal O^G(X; \mathcal A)$, here we are viewing objects $A$ and $B$ as objects in $\mathcal O^G(X; \mathcal A)$ by restriction to $G\times X\times[1, \infty)$. The map $\phi\rightarrow\sum_{f\in F}\phi^f\cdot f$ then gives a one-one correspondence between  morphisms of these two category. 

\begin{thm}\label{equivalence} Let $\Gamma=G\rtimes_\alpha F$ be as before. Then for any $\Gamma$- CW complex $X$ and any additive category (with involution) $\mathcal A$ with a right $\Gamma$-action, the two additive categories (with involutions) $\mathcal O^{\Gamma}(X;\mathcal A)$ and $\mathcal O^G(X; \mathcal A)_{\alpha}[F]$ are equivalent provided $\mathcal E^X_{\Gamma cc}\subseteq\mathcal E^X_{Gcc}$. In particular, they are equivalent if $F$ is finite or if the action of $\Gamma$ on $X$ is free. 
\end{thm}
\begin{proof}\label{eq} Define a functor $\tau: \mathcal O^G(X; \mathcal A)_{\alpha}[F]\rightarrow \mathcal O^{\Gamma}(X;\mathcal A)$ as follows:
\vskip 5 pt
\noindent
\textit{on objects:} $$(\tau A)_{(\r, x, s)}:=f^*(A_{(\alpha_f(g), fx, s)})=(\alpha^fA)_{(g, x, s)}$$
 where $(\r, x, s)\in\Gamma\times X\times[1, \infty)$ and $g\in G, f\in F$ are the unique elements so that $\r=gf^{-1}$. One can easily check that $(\tau A)_{(\r, x, s)}=(\r^{-1})^*(A_{(e, \r^{-1}x, s)})$, from which we conclude $\tau A$ is $\Gamma$-invariant.

\vskip 5 pt
\noindent
\textit{on morphisms:} Let $\phi=\ds\sum_{f\in F}\phi^f\cdot f: A\rightarrow B$ be a morphism in $\mathcal O^G(X; \mathcal A)_{\alpha}[F]$, where $\phi^f: A\rightarrow\alpha^fB$ is a morphism in $\mathcal O^G(X; \mathcal A)$. By definition, there are only finitely many $f\in F$ with $\phi^f\ne 0$.  

Define $$\tau \phi: \tau A\rightarrow \tau B$$ to be the sum $\ds\sum_{f\in F}\tau_f\phi^f$,  where
$$\tau_f\phi^f: \tau A\rightarrow\tau B$$ is defined as follows:
$$(\tau_f\phi^f)_{(\r', x', s'), (\r, x, s)}:=\begin{cases} 0 &\mbox{if } \r^{-1}\r'\notin Gf^{-1}\\
(\r^{-1})^{*}(\phi^f_{(\r^{-1}\r'f, \r^{-1}x', s'),(e, \r^{-1}x,s)}) & \mbox{if } \r^{-1}\r'\in Gf^{-1}
\end{cases}$$ 

\vskip 10 pt
We have to check $(\r^{-1})^{*}\big(\phi^f_{(\r^{-1}\r'f, \r^{-1}x', s'),(e, \r^{-1}x,s)}\big)$ is indeed a map from
$(\tau A)_{(\r, x, s)}$ to $(\tau B)_{(\r', x', s')}$.
But this is because 
$$(\tau A)_{(\r, x, s)}=(\r^{-1})^*(A_{(e, \r^{-1}x, s)})$$
$$ (\tau B)_{(\r', x', s')}=(\r^{-1})^*\big((\tau B)_{(\r^{-1}\r', \r^{-1}x', s')}\big)=(\r^{-1})^*\big((\alpha^fB)_{(\r^{-1}\r'f, \r^{-1}x', s')}\big)$$
and $$\phi^f_{(\r^{-1}\r'f, \r^{-1}x', s'),(e, \r^{-1}x,s)}: A_{(e, \r^{-1}x,s)}\longrightarrow(\alpha^fB)_{(\r^{-1}\r'f, \r^{-1}x', s')}$$

\vskip 10pt

The formulas above are motivated by the intuition explained in the paragraph preceding this theorem. We firstly show $\tau$ is well-defined, i.e. it is a genuine additive functor and respects
control conditions on objects and morphisms. We then show $\tau$ actually gives an equivalence of two additive categories. In what follows,  we will again omit
the $s$-component in appropriate places.
\vskip 10 pt
\noindent
(i) $\tau A$ is $\Gamma$-invariant: already checked.

\vskip 10 pt
\noindent
(ii) $\tau\phi$ is $\Gamma$-invariant: it suffices to check for each $f\in F$, $\tau_f\phi^f$ is $\Gamma$-invariant. For any $l\in\Gamma$, we have
$$[l^*(\tau_f \phi^f)]_{(\r', x'), (\r,x)}=l^*[(\tau_f \phi^f)_{(l\r',lx'), (l\r, lx)}]$$
If $\r^{-1}\r'\notin Gf^{-1}$, then $(l\r)^{-1}l\r'\notin Gf^{-1}$
, hence $[l^*(\tau_f \phi^f)]_{(\r', x'), (\r,x)}=0=(\tau_f \phi^f)_{(\r', x'), (\r,x)}$.\\
If $\r^{-1}\r'\in Gf^{-1}$, then 
\begin{align*}
[l^*(\tau_f\phi^f)]_{(\r', x'), (\r,x)}&=l^*[(\tau_f\phi^f)_{(l\r',lx'), (l\r, lx)}]\\
&=l^*[(\r^{-1}l^{-1})^*(\phi^f_{(\r^{-1}\r'f, \r^{-1}x'), (e,\r^{-1}x)})]\\
&=(\r^{-1})^*(\phi^f_{(\r^{-1}\r'f, \r^{-1}x'), (e,\r^{-1}x)})\\
&=(\tau_f\phi^f)_{(\r', x'), (\r, x)}\\
\end{align*}
Thus $l^*(\tau_f\phi^f)=\tau_f\phi^f$, for all $f\in F$ and $l\in\Gamma$. This shows $\tau\phi$ is $\Gamma$-invariant.

\vskip 10pt
\noindent
(iii) $\tau (id_A)=id_{\tau A}$: note $id_A: A\rightarrow A$ in $\mathcal O^G(X; \mathcal A)_{\alpha}[F]$ is given by $id_A\cdot e$. Hence 
\begin{align*}[\tau(id_A)]_{(\r',x'),(\r,x)}
&=\begin{cases} 0 &\mbox{if } \r^{-1}\r'\notin G \\
(\r^{-1})^{*}[(id_{A})_{(\r^{-1}\r', \r^{-1}x'),(e, \r^{-1}x)}] & \mbox{if } \r^{-1}\r'\in G \end{cases} \\
&=\begin{cases} 0 &\mbox{if } (\r,x)\neq (\r', x') \\
(id_{\tau A}){(\r',x'),(\r,x)} & \mbox{if } (\r,x)=(\r',x')
 \end{cases} 
\end{align*}

\vskip 10pt
\noindent
(iv) $\tau (\psi\circ\phi)=(\tau\psi)\circ(\tau\phi)$: Let $\phi=\sum_{f\in F}\phi^f\cdot f: A\rightarrow B$ and $\psi=\sum_{h\in F}\psi^h\cdot h: B\rightarrow C$ be two morphisms in $\mathcal O^G(X;\mathcal A)_\alpha[F]$.  We have
$$\tau(\psi\circ\phi)=\tau\Big(\sum_{k\in F}\big(\sum_{hf=k}(\alpha^f\psi^h)\circ\phi^f\big)\cdot k\Big)=\sum_{k\in F}\sum_{hf=k}\tau_k\big((\alpha^f\psi^h)\circ\phi^f\big)$$
and
$$\tau(\psi)\circ\tau(\phi)=\big(\sum_{h\in F}\tau_h\psi^h\big)\circ\big(\sum_{f\in F}\tau_f\phi^f\big)=\sum_{k\in F}\sum_{hf=k}(\tau_h\psi^h)\circ(\tau_f\phi^f)$$
So it suffices to show $\tau_k\big((\alpha^f\psi^h)\circ\phi^f\big)=(\tau_h\psi^h)\circ(\tau_f\phi^f)$ when $k=hf$.

Now by definition,  when $\r^{-1}\r'\notin Gk^{-1}$,  we have  $$[\tau_k\big((\alpha^f\psi^h)\circ\phi^f\big)]_{(\r', x'), (\r,x)}=0$$ and 
$$ [(\tau_h\psi^h)\circ(\tau_f\phi^f)]_{(\r', x'), (\r, x)}=(\tau_h\psi^h)_{(\r', x'), (\r'', x'')}\circ(\tau_f\phi^f)_{(\r'', x''), (\r, x)}$$
where on the right hand side of the above identity, we are taking sum over $(\r'', x'')\in\Gamma\times X$.  However a term in this sum can be non-zero only when $\r''^{-1}\r'\in Gh^{-1}$ and $\r^{-1}\r''\in Gf^{-1}$, which implies when $\r^{-1}\r'\notin Gk^{-1}$, it must be zero. This verifies 
$[\tau_k\big((\alpha^f\psi^h)\circ\phi^f\big)]_{(\r', x'), (\r,x)}= [(\tau_h\psi^h)\circ(\tau_f\phi^f)]_{(\r', x'), (\r, x)}$ when $\r^{-1}\r'\notin Gk^{-1}$.

When $\r^{-1}\r'\in Gk^{-1}$, on one hand, we have (in the following identities, we are taking sum over $(g'', x'')\in G\times X$.)
\begin{align*}&\ \ \ \ [\tau_k\big((\alpha^f\psi^h)\circ\phi^f\big)]_{(\r', x'), (\r,x)}\\
&=(\r^{-1})^*\big[\big(\alpha^f\psi^h)\circ\phi^f\big)_{(\r^{-1}\r'k, \r^{-1}x'), (e, \r^{-1}x)}\big]\\
&=(\r^{-1})^*[(\alpha^f\psi^h)_{(\r^{-1}\r'k, \r^{-1}x'), (g'', x'')}\circ\phi^f_{(g'',x''), (e, \r^{-1}x)}]\\
&=(\r^{-1})^*\big[f^*(\psi^h_{(\alpha_f(\r^{-1}\r'k), f\r^{-1}x'), (\alpha_f(g''), fx'')})\circ\phi^f_{(g'',x''), (e, \r^{-1}x)}]
\\
&=(\r^{-1})^*\big[\big(\alpha_f(g''^{-1})f\big)^*\big(\psi^h_{(\alpha_f(g''^{-1}\r^{-1}\r'k), \alpha_f(g''^{-1})f\r^{-1}x'), (e, \alpha_f(g''^{-1})fx'')}\circ\phi^f_{(g'',x''), (e, \r^{-1}x)}]\big)\big]\\
&=(\r^{-1})^*\big[\big(fg''^{-1}\big)^*\big(\psi^h_{(\alpha_f(g''^{-1}\r^{-1}\r'k), fg''^{-1}\r^{-1}x'), (e, fg''^{-1}x'')}\circ\phi^f_{(g'', x''), (e, \r^{-1}x)}\big)\big]\\
&=\big(fg''^{-1}\r^{-1}\big)^*\big(\psi^h_{(\alpha_f(g''^{-1}\r^{-1}\r'k), fg''^{-1}\r^{-1}x'), (e, fg''^{-1}x'')}\big)\circ\big(\r^{-1}\big)^*\big(\phi^f_{(g'', x''), (e, \r^{-1}x)}\big)\\
&=(\tau_h\psi^h)_{(\r g''f^{-1}\alpha_f(g''^{-1}\r^{-1}\r'k)h^{-1},\r g''f^{-1}fg''^{-1}\r^{-1}x'),(\r g''f^{-1}, \r g''f^{-1}fg''^{-1}x'')}\circ(\tau_f\phi^f)_{(\r g''f^{-1}, \r x''), (\r, x)}\\
&=(\tau_h\psi^h)_{(\r' kf^{-1}h^{-1}, x'),(\r g''f^{-1}, \r x'')}\circ(\tau_f\phi^f)_{(\r g''f^{-1}, \r x''), (\r, x)}\\
&=(\tau_h\psi^h)_{(\r', x'), (\r g''f^{-1}, \r x'')}\circ(\tau_f\phi^f)_{(\r g''f^{-1}, \r x''), (\r, x)}\\
&=(\tau_h\psi^h)_{(\r', x'), (\r g''f^{-1}, x'')}\circ(\tau_f\phi^f)_{(\r g''f^{-1}, x''), (\r, x)}
\end{align*}
\vskip 10pt
\noindent
The last identity holds because when $x''$ runs over $X$, $\r x''$ also runs over $X$.

Now on the other hand, we have
$$[(\tau_h\psi^h)\circ(\tau_f\phi^f)]_{(\r', x'), (\r, x)}=
(\tau_h\psi^h)_{(\r', x'),(\r'', x'')}\circ(\tau_f\phi^f)_{(\r'', x''),(\r,x)}
$$
Note we are taking sum over $(\r'', x'')\in\Gamma\times X$  on the right hand of the above identity. However only those terms with  $\r^{-1}\r''\in Gf^{-1}$ can be non-zero (note $\r^{-1}\r''\in Gf^{-1}, \r^{-1}\r'\in Gk^{-1}$ together with $k=hf$ imply $\r''^{-1}\r'\in Gh^{-1}$). So we only have to take sum over those terms with  $\r''=\r g''f^{-1}, g''\in G$. 
Hence we have 
\begin{align*}&\ \ \ \ [(\tau_h\psi^h)\circ(\tau_f\phi^f)]_{(\r', x'), (\r, x)}\\
&=(\tau_h\psi^h)_{(\r', x'),(\r g''f^{-1}, x'')}\circ(\tau_f\phi^f)_{(\r g''f^{-1}, x''),(\r,x)}\end{align*}
Therefore $\tau_k\big((\alpha^f\psi^h)\circ\phi^f\big)=(\tau_h\psi^h)\circ(\tau_f\phi^f)$ when $k=hf$. Thus $\tau (\psi\circ\phi)=(\tau\psi)\circ(\tau\phi)$.

\vskip 10pt
\noindent
(v) Additivity: this is easy.

\vskip 10pt
\noindent
(vi) Object support: $(\r,x, s)\in$supp$(\tau A)$ if and only if $(e, \r^{-1}x, s)\in$supp$(A)$. By definition, there is a compact set $K\subseteq X$, so that $\r^{-1}x\in K$ for all $(e, \r^{-1}x, s)\in$supp$(A)$. Hence 
$\r^{-1}x\in K$ for all $(\r, x, s)\in$supp$(\tau A)$. This shows $\tau A$ respects the object support condition.

\vskip 10pt
\noindent
(vii) Morphism support: we firstly show each $\tau_f\phi^f, f\in F$ respects the morphism support condition.  By definition of $\tau_f\phi^f$, its morphism support condition in the $\Gamma$-direction is easily seen to be
satisfied. Let us focus on its morphism support condition in the $X\times [1, \infty)$-direction.  By definition, for $\phi^f: A\rightarrow\alpha^fB$, there exists $E\in\mathcal E^X_{Gcc}$,  such that the projection of supp$(\phi^f)$ in $(X\times [1,\infty))^2$ is contained in $E$. By definition of $\tau_f\phi^f$ and the fact that supp$(A)$ and supp$(\alpha^fB)$ are  $G$-cocompact in the $G\times X$-direction, there exists a compact subset $K\subseteq X$, such that the projection of supp$(\tau_f\phi^f)$ in $(X\times[1, \infty))^2$ is contained in 
$$E'=\{(x, s)\times(x',s')|\ \exists \r, \r'\in\Gamma\ \text{s.t.}\ \r^{-1}\r'\in Gf^{-1}, \r^{-1}x, f^{-1}\r'^{-1}x'\in K, (\r^{-1}x, s)\times(\r^{-1}x', s')\in E \}$$
Replacing $\r'$ by $\r'f^{-1}$, one sees
$$E'=\{(x, s)\times(x',s')|\ \exists \r, \r'\in\Gamma\ \text{s.t.}\ \r^{-1}\r'\in G, \r^{-1}x, \r'^{-1}x'\in K, (\r^{-1}x, s)\times(\r^{-1}x', s')\in E \}$$
But by Lemma \ref{control}, $E'\in\mathcal E^X_{\Gamma cc}$. This shows, for each $f\in F$, $\tau_f\phi^f$ respects the morphism support condition for morphisms in $\mathcal O^\Gamma(X; \mathcal A)$. Therefore $\tau\phi=\ds\sum_{f\in F}\tau_f\phi^f$
respects the morphism support condition in $\mathcal O^\Gamma(X; \mathcal A)$ since supp($\tau\phi$) is contained in the finite union $\bigcup_{f\in F}\text{supp}(\tau_f\phi^f)$. 

\vskip 10pt
We now complete the  verification work  that $\tau$ is a well-defined functor from $\mathcal O^G(X; \mathcal A)_{\alpha}[F]$ to $\mathcal O^{\Gamma}(X;\mathcal A)$. Next we show it is an equivalence of additive categories (with involution).  
\vskip 10pt
\noindent
(a) $\tau$ is full on objects: for any object $A$ in $\mathcal O^{\Gamma}(X;\mathcal A)$, let Res$A$ denote its restriction to $G\times X\times[1, \infty)$, i.e. (Res$A)_{(g, x, s)}:=A_{(g, x, s)}, (g, x, s)\in G\times X\times [1, \infty)$. Clearly 
Res$A$ is an object in $\mathcal O^G(X; \mathcal A)_{\alpha}[F]$. One also easily sees that $\tau(\text{Res}A)=A$. 

\vskip 10pt
\noindent
(b) $\tau$ is full on morphisms: for any objects $A, B$ in  $\mathcal O^G(X; \mathcal A)_{\alpha}[F]$, morphism $\Phi: \tau A\rightarrow \tau B$ in $\mathcal O^{\Gamma}(X;\mathcal A)$, and $f\in F$, define
$$\phi^f: A\rightarrow\alpha^fB$$ by defining
$$\phi^f_{(g', x', s'), (g, x,s)}:=\Phi_{(g'f^{-1}, x',s'), (g, x,s)}$$
Note since
$$\Phi_{(g'f^{-1}, x',s'), (g, x,s)}: (\tau A)_{(g, x, s)}=A_{(g, x, s)}\longrightarrow(\tau B)_{(g'f^{-1}, x', s')}=(\alpha^fB)_{(g', x', s')}$$
we see that $\phi^f_{(g', x', s'), (g, x,s)}$ is indeed a map from $A_{(g, x, s)}$ to $(\alpha^fB)_{(g', x', s')}$.

One easily checks $\phi^f$ is $G$-invariant.  We now check $\phi^f$ respects the morphism support condition for morphisms in $\mathcal O^G(X; \mathcal A)$. By definition, the projection of supp($\phi^f$) in the $X\times [1, \infty)$-direction is the same as the projection of supp($\Phi$) in the $X\times[1, \infty)$-direction. Now by assumption,  $\mathcal E^X_{\Gamma cc}\subseteq\mathcal E^X_{Gcc}$.  Therefore  $\phi^f$ respects the morphism support condition in the $X\times[1,\infty)$-direction.  By the morphism control condition of $\Phi$ in the $\Gamma$-direction, there is a finite set $S\subseteq\Gamma$ with the property that  $\Phi_{(g'f^{-1}, x',s'), (g, x,s)}\ne 0$ implies $g^{-1}g'f^{-1}\in S$. Therefore 
$\phi^f_{(g', x', s'), (g, x,s)}\ne 0$ implies $g^{-1}g'f^{-1}\in S$. This firstly shows  $\phi^f$ respects the morphism support condition in the $G$-direction. Therefore $\phi^f$ is a morphism in $\mathcal O^G(X; \mathcal A)$. Secondly, since $S$ is finite, its image in $F$ under the natural projection $\Gamma\rightarrow F$ is finite. This implies there are only finitely many $f\in F$ with the property that there exist 
$g, g'\in G$ so that $g^{-1}g'f^{-1}\in S$. Hence there are only finitely many $f\in F$ so that $\phi^f\ne 0$.
Therefore 
$$\phi=\ds\sum_{f\in F}\phi^f\cdot f: A\rightarrow B$$ 
defines a morphism in $\mathcal O^G(X; \mathcal A)_{\alpha}[F]$.

We now show $\tau\phi=\Phi$.  For any $(\r', x',s'), (\r, x,s)\in \Gamma\times X\times[1, \infty)$, let $\r^{-1}\r'=gf^{-1}, g\in G, f\in F$.  We then have 
\begin{align*}
(\tau\phi)_{(\r', x',s'), (\r, x,s)}&=(\tau_{f}\phi^f)_{(\r', x', s'),(\r,x,s)}\\
&=(\r^{-1})^*[\phi^f_{(g, \r^{-1}x', s'), (e, \r^{-1} x, s)}]\\
&=(\r^{-1})^*[\Phi_{(gf^{-1}, \r^{-1}x',s'), (e, \r^{-1}x, s)}]\\
&=\Phi_{(\r', x', s'), (\r, x, s)}
\end{align*}
This shows $\tau\phi=\Phi$ and completes the proof that $\tau$ is full on morphisms.

\vskip 10pt
\noindent
(c) $\tau$ is faithful on morphisms: this is easy using the identity $\phi^f_{(g', x', s'), (g, x,s)}:=(\tau \phi)_{(g'f^{-1}, x',s'), (g, x,s)}$.

\vskip 10pt
Therefore $\tau$ is an equivalence between $\mathcal O^G(X; \mathcal A)_\alpha[F]$ and $\mathcal O^{\Gamma}(X; \mathcal A)$  provided $\mathcal E^X_{\Gamma cc}\subseteq\mathcal E^X_{Gcc}$. Now by Lemma \ref{11},  $\mathcal E^X_{\Gamma cc}\subseteq\mathcal E^X_{Gcc}$ if $F$ is finite or the action of $\Gamma$ on $X$ is free. This completes the proof of the theorem.
\end{proof}

\begin{rem} One sees from the proof that the assumption $\mathcal E^X_{\Gamma cc}\subseteq\mathcal E^X_{Gcc}$ is only used to show the functor $\tau$ is full. So $\tau$ is well-defined even without this assumption.
\end{rem}

\section{Proof of the Main theorems}\label{main}
In this section, we prove Theorem A and Theorem B. We firstly prove a lemma on $L$-groups which is another ingredient in proving Theorem A.

\subsection{A Lemma on $L$-groups} In this subsection, we prove Lemma \ref{l group}.  We will need the following well-known fact: for every  group $G$ and every orientation map $w: G\longrightarrow\{\pm 1\}$, if $Wh(G)=0, \tilde{K_0}(\mathbb Z[G])=0$ and $K_i(\mathbb Z[G])=0, i\le -1$,  then all the $L$-groups of $\mathbb Z[G]$ associated to $w$ with various decorations are naturally isomorphic (we will omit $w$ from the notations):
$$L^s_n(\mathbb Z[G])\cong L^h_n(\mathbb Z[G])=L^1_n(\mathbb Z[G])\cong L^0_n(\mathbb Z[G])\cong L^{<-1>}_n(\mathbb Z[G])\cong\cdots\cdots \cong L^{<-\infty>}_n(\mathbb Z[G])$$
where $Wh(G), \tilde{K_0}(\mathbb Z[G])$ and $K_i(\mathbb Z[G],  i\le -1$ are the Whitehead group, 
the reduced $K_0$-group and the negative $K$-groups of $\mathbb Z[G]$  respectively. For an explanation about the decorations of  various $L$-groups and a proof of the above fact, see \cite[Remmark 21 on page 720 and Proposition 23 on page 721]{LR}. 

\begin{lem}\label{l group}  Let $G$ be a group. Assume $Wh(G)=0, \tilde{K_0}(\mathbb Z[G])=0$ and $K_i(\mathbb Z[G])=0, i\le -1$. Then for every group of the form $\Gamma=G\rtimes_{\alpha}\mathbb Z$ and every orientation map $w: \Gamma\longrightarrow\{\pm\}$, 
its simple $L$-groups and ultimate $L$-groups are naturally isomorphic, i.e. $L^s_n(\mathbb Z[\Gamma])\cong L^{<-\infty>}_n(\mathbb Z[\Gamma]), \forall n\in\mathbb Z$, naturally. If in addition $Wh(G\times\mathbb Z)=0$, then $L^h_n(\mathbb Z[\Gamma])$
adds to the natural isomorphisms, i.e. $L^s_n(\mathbb Z[\Gamma])\cong L^h_n(\mathbb Z[\Gamma])\cong L^{<-\infty>}_n(\mathbb Z[\Gamma]), \forall n\in\mathbb Z$, naturally.
\end{lem}

\begin{proof}When $\alpha: G\longrightarrow G$ is trivial, then the result follows immediately from the above well-known fact and the \textit{Shaneson splitting}:
$$L^{<j>}_n(\mathbb Z[G\times\mathbb Z])\cong L_n^{<j>}(\mathbb Z[G])\oplus L_{n-1}^{<j-1>}(\mathbb Z[G]), j\le 1$$
$$L^{<s>}_n(\mathbb Z[G\times\mathbb Z])\cong L_n^{<s>}(\mathbb Z[G])\oplus L_{n-1}^{h}(\mathbb Z[G])$$
which is natural with respect to the natural forgetful maps between them indexed by $s\rightarrow h\rightarrow 0\rightarrow\cdots\rightarrow -\infty$. 

For the general case, we need to make use of the results of Ranicki \cite{RAA1}.
Firstly, there is a long exact sequence \cite[Page 413]{RAA1}:
\begin{align}\label{les1}
\cdots\longrightarrow L^s_n(\mathbb Z[G])\longrightarrow L_n^s(\mathbb Z[\Gamma])\longrightarrow L'_{n-1}(\mathbb Z[G])\longrightarrow L^s_{n-1}(\mathbb Z[G])\longrightarrow\cdots
\end{align}
where $L'_{n-1}(\mathbb Z[G])$ is a certain intermediate $L$-group with torsion lies in $(1-\alpha_*)^{-1}(\overline{G})$, where 
$$1-\alpha_*: \tilde{K}_1(\mathbb Z[G])\longrightarrow \tilde{K}_1(\mathbb Z[G])$$
is the map induced by $\alpha$ and $\overline{G}$ is the image of $G$ in $\tilde{K}_1(\mathbb Z[G])$. 
Now by assumption, $Wh(G)=0$, this implies $\overline{G}=\tilde{K}_1(\mathbb Z[G])$. Hence $(1-\alpha_*)^{-1}(\overline{G})=\tilde{K}_1(\mathbb Z[G])$. Therefore $L'_{n}(\mathbb Z[G])=L^h_{n}(\mathbb Z[G]), \forall n\in\mathbb Z$ and \ref{les1} becomes:
\begin{align}\label{les2}
\cdots\longrightarrow L^s_n(\mathbb Z[G])\longrightarrow L_n^s(\mathbb Z[\Gamma])\longrightarrow L^h_{n-1}(\mathbb Z[G])\longrightarrow L^s_{n-1}(\mathbb Z[G])\longrightarrow\cdots
\end{align}
Now there is also a well-known Wang type long exact sequence for the ultimate $L$-groups (\cite[the proof of Lemma 4.2]{Luck1}):
\begin{align}\label{les3}
\cdots\longrightarrow L^{<-\infty>}_n(\mathbb Z[G])\longrightarrow L_n^{<-\infty>}(\mathbb Z[\Gamma])\longrightarrow L^{-\infty}_{n-1}(\mathbb Z[G])\longrightarrow L^{<-\infty>}_{n-1}(\mathbb Z[G])\longrightarrow\cdots
\end{align}
%This can be derived from repeatedly applying \cite[Theorems 5.1-5.3]{RAA1}.
The natural forgetful maps between these groups give rise to a commutative diagram of the above long exact sequences. By assumption, all these forgetful maps are isomorphisms for the $L$-groups of $\mathbb Z[G]$. Therefore, by the five lemma, $L^s_n(\mathbb Z[\Gamma])\cong L_n^{<-\infty>}(\mathbb Z[\Gamma])$.  

If in addition $Wh(G\times\mathbb Z)=0$, then by replacing $G$ by $G\times\mathbb Z$  in \ref{les1}, where we use the trivial action of $\mathbb Z$ on $\mathbb Z$ (so ($G\times\mathbb Z)\rtimes\mathbb Z\cong\Gamma\times\mathbb Z$),  and use the same argument as we did for $G$, we get
\begin{align}\label{les4}
\cdots\longrightarrow L^s_n(\mathbb Z[G\times\mathbb Z])\longrightarrow L_n^s(\mathbb Z[\Gamma\times\mathbb Z])\longrightarrow L^h_{n-1}(\mathbb Z[G\times\mathbb Z])\longrightarrow L^s_{n-1}(\mathbb Z[G\times\mathbb Z])\longrightarrow\cdots
\end{align}
By Shaneson splitting, the above sequence naturally splits into two long exact sequences, one of which is
\ref{les2} and another one is:
\begin{align}\label{les5}
\cdots\longrightarrow L^h_{n-1}(\mathbb Z[G])\longrightarrow L_{n-1}^h(\mathbb Z[\Gamma])\longrightarrow L^0_{n-2}(\mathbb Z[G])\longrightarrow L^0_{n-2}(\mathbb Z[G])\longrightarrow\cdots
\end{align}
Now as before, \ref{les2}, \ref{les3}, and \ref{les5} together imply the natural isomorphisms  $L^s_n(\mathbb Z[\Gamma])\cong L^h_n(\mathbb Z[\Gamma])\cong L^{<-\infty>}_n(\mathbb Z[\Gamma])$.
\end{proof}

\begin{rem} We cannot conclude isomorphisms for the $L$-groups of $\mathbb Z[\Gamma]$ with other decorations. There is the issue of intermediate $L$-groups. But if we assume $Wh(G\times\mathbb Z^n)=0, \forall n\in\mathbb N$, then all the $L$-groups of $\mathbb Z[\Gamma]$ are naturally isomorphic. 
\end{rem}

\subsection{Proof of Theorem B} In this subsection, we  prove Theorem B first. Theorem A will be a consequence of Theorem B and Lemma \ref{l group}. We will prove the more general version of Theorem B
as explained in Remark \ref{general}.  We firstly treat the $L$-theory case.  Let $\mathcal A$ be a right $G\rtimes_\beta\mathbb Z$ additive category with involution.  View it as a $G$-additive category in the natural way. Since $G$ is torsion free, $G\rtimes_\beta\mathbb Z$ is also torsion free.  This implies the $L$-theoretic assembly map $A_L^{\mathcal A}$ for $G\rtimes_{\beta}\mathbb Z$ in \ref{l assembly} is equivalent to the following assembly map:
$$H^{G\rtimes_\beta\mathbb Z}_n(E(G\rtimes_\beta\mathbb Z); \textbf{L}^{<-\infty>}_{\mathcal A})\rightarrow L^{<-\infty>}_n(\mathcal A_\alpha[G\rtimes_\beta\mathbb Z])$$
where $E(G\rtimes_\beta\mathbb Z)$ is the classifying space for free $G\times_\beta\mathbb Z$-actions.
See \cite[Proposition 66 on page 743]{LR}. 
Compare \cite[Theorem 8.14]{BFL}.  Therefore, by Theorems \ref{obstruction1}, \ref{vanishing}, $A_L^{\mathcal A}$ for $G\times_\beta\mathbb Z$ is an isomorphism if and only if 
$$L^{<-\infty>}_n(\mathcal O^{G\rtimes_\beta\mathbb Z}(E(G\rtimes_\beta\mathbb Z), \mathcal A))=0, \forall n\in\mathbb Z$$ 
Now by Theorem \ref{equivalence}, we have an equivalence of additive categories with involutions
$$\mathcal O^{G\rtimes_\beta\mathbb Z}(E(G\rtimes_\beta\mathbb Z), \mathcal A)\cong\mathcal O^G(E(G\rtimes_\beta\mathbb Z), \mathcal A)_\beta[\mathbb Z]$$
since the action of $G\rtimes_\beta\mathbb Z$ on $E(G\rtimes_\beta\mathbb Z)$ is free.  Denote
$\mathcal O^G(E(G\rtimes_\beta\mathbb Z), \mathcal A)$ by $\tilde{\mathcal A}$. Note that by the characterization properties of classifying spaces, $E(G\rtimes_\beta\mathbb Z)$ is also a model for  the classifying space for free $G$-actions. Now by assumption, the $L$-theoretic FJC holds for $G$ with coefficient in $\mathcal A$, therefore, by Theorems \ref{obstruction1}, \ref{vanishing}, we have 
$$L^{<-\infty>}_n(\tilde{\mathcal A})=0, \forall n\in\mathbb Z$$
Now by applying the Wang type long exact sequence \ref{les3} to $L^{<-\infty>}_*(\tilde{\mathcal A}_\beta[\mathbb Z])$, which is well-known when the coefficient is a ring and can be extended to with twisted coefficient in any additive category (see \cite[proof of Theorem 8.14]{BFL}), we immediately get that 
$$L^{<-\infty>}_n(\tilde{\mathcal A}_\beta[\mathbb Z])=0,\forall n\in\mathbb Z$$
This implies 
$$L^{<-\infty>}_n(\mathcal O^{G\rtimes_\beta\mathbb Z}(E(G\rtimes_\beta\mathbb Z), \mathcal A))=0, \forall n\in\mathbb Z$$
Therefore the $L$-theoretic FJC with coefficient in $\mathcal A$ holds for $G\rtimes_\beta\mathbb Z$. This proves part (1) of Theorem B. 

We now turn into part (2) of Theorem B. When $R$ is regular and $G\rtimes_\beta\mathbb Z$ is torsion free, the $K$-theoretic assembly map $A_K^R$ for $G\rtimes_\beta\mathbb Z$ in \ref{k assembly} is equivalent to 
the following assembly map:
$$H^{G\rtimes_\beta\mathbb Z}_n(E(G\rtimes_\beta\mathbb Z); \textbf{K}_{\mathcal A})\rightarrow K_n(\mathcal A_\alpha[G\rtimes_\beta\mathbb Z])$$
Similar to the argument as in the $L$-theoretic case, we see that the above assembly map is an isomorphism if and only if 
$$K_n(\mathcal O^G(E(G\rtimes_\beta\mathbb Z), R)_\beta[\mathbb Z])=0, \forall n\in\mathbb Z$$
Now, there is a Wang type long exact sequence for the non-connective $K$-theory of $\mathcal O^G(E(G\rtimes_\beta\mathbb Z), R)_\beta[\mathbb Z]$. This is well-known when the coefficient is a ring and has been recently generalized by L\"uck-Steimle \cite[Theorem 0.1, Remark 0.2]{LS} to with coefficient in any additive category. Using this sequence and the assumption that the $K$-theoretic FJC holds for $G$ with coefficient in $R$, we see that
$$K_n(\mathcal O^G(E(G\rtimes_\beta\mathbb Z), R)_\beta[\mathbb Z])\cong NK_n(\mathcal O^G(E(G\rtimes_\beta\mathbb Z), R);\beta)\oplus NK_n(\mathcal O^G(E(G\rtimes_\beta\mathbb Z), R);\beta)$$
Denote the Nil-groups $NK_n(\mathcal O^G(E(G\rtimes_\beta\mathbb Z), R);\beta)$ as defined in \cite{LS}
by $Nil_{n, R}^{G\rtimes_\beta\mathbb Z}$, then 
$$K_n(\mathcal O^G(E(G\rtimes_\beta\mathbb Z), R)_\beta[\mathbb Z])=0$$ if and only if
$Nil_{n, R}^{G\rtimes_\beta\mathbb Z}=0$. This proves part (2) of Theorem B.

\subsection{Proof of Theorem A} We now prove Theorem A. Let us assume $G\in\mathcal{FJ}$.  If $G$ is not torsion free, then there is nothing to prove, since it cannot be realized as the fundamental group of a closed aspherical manifold. So let us assume $G$ is torsion free. Since $G$ satisfies the $K$-theoretic FJC with coefficient in $\mathbb Z$, it follows that 
$$Wh(G)=0,\ \ \tilde{K_0}(\mathbb Z[G])=0,\ \ K_i(\mathbb Z[G])=0, i\le -1$$
See \cite[Conjecture 1 on page 708, Conjecture 3 on page 710]{LR}.  Therefore Lemma \ref{l group} applies and we have, for every orientation map $w: G\rtimes\mathbb Z\longrightarrow\{\pm\}$, natural isomorphisms
\begin{align}\label{s-infinity}
L^s_n(\mathbb Z[G\rtimes\mathbb Z])\cong L^{<-\infty>}_n(\mathbb Z[G\rtimes\mathbb Z]), \forall n\in\mathbb Z
\end{align}
Because $G$ is torsion free and satisfies the $L$-theoretic FJC with coefficient in $\mathbb Z$, Theorem B applies and the $L$-theoretic FJC with coefficient in $\mathbb Z$ holds for $G\rtimes\mathbb Z$. 
Therefore, by a standard surgery long exact sequence argument, we see that the simple Borel conjecture holds for $G\rtimes\mathbb Z$. For the convenience of the reader, let us sketch the main idea, more details can be found in \cite[Theorem 28 on page 723]{LR},  \cite[pages 26-28]{FRR} .
Let $M$ be a closed aspherical manifold of dimension $n\ge 5$ with $\pi_1(M)\cong G\rtimes\mathbb Z$.
Let $S^{s}(M)$ be its simple topological structure set (an abelian group indeed). It consists of the equivalence classes of all simple homotopy equivalence $f: N\longrightarrow M$,  from another closed  manifold $N$ to $M$. Two such maps $f: N\longrightarrow M,\ f': N'\longrightarrow M$ are equivalent if there is a homeomorphism $g: N\longrightarrow N'$ so that $f'\circ g$ is homotopic to $f$.  Therefore the simple Borel conjecture holds for $G$ if and only if $S^s(M)$ consists of one point. Now there is a surgery
long exact sequence for $S^s(M)$:
$$\xymatrix{
\cdots\ar[r]&\mathcal N_{n+1}(M)\ar[r]^{\sigma_{n+1}\ \ \ \ }&L^s_{n+1}(\mathbb Z[G\rtimes\mathbb Z])\ar[r]^{\ \ \ \ \ \partial_{n+1}}&S^s(M)\ar[r]^{\eta_n\ }&\mathcal N_n(M)\ar[r]^{\sigma_n\ \ \ \ }& L^s_n(\mathbb Z[G\rtimes\mathbb Z])
}
$$
For each $i\ge n$, the assembly map \ref{l assembly} and the above sequence fit into the following commutative diagram
$$
\xymatrix{ 
\mathcal N_i(M) \ar[r]^{\sigma_i}\ar[d]^{p_i} & L^s_i(\mathbb Z[G\rtimes\mathbb Z])\ar[d]^{id}\\
H_i(M; \textbf{L}^s_{\mathbb Z})\ar[d]^{q_i}\ar[r] & L_i^s(\mathbb Z[G\rtimes\mathbb Z])\ar[d]^{f_i}\\
H_i^{G\rtimes\mathbb Z}(E(G\rtimes\mathbb Z); \textbf{L}^{<-\infty>}_\mathbb Z)\ar[r]^{\ \ \ \ A_L^{\mathbb Z}} & L^{<-\infty>}_i(\mathbb Z[G\rtimes\mathbb Z])
}
$$
We have showed that $f_i$ and $A^{\mathbb Z}_L$ are isomorphisms for all $i\ge n$. 
$q_i$ is also an isomorphism since 
$$H_i^{G\rtimes\mathbb Z}(E(G\rtimes\mathbb Z); \textbf{L}^{<-\infty>}_\mathbb Z)\cong H_i(M; \textbf{L}_\mathbb Z^{<-\infty>})\cong H_i(M; \textbf L^s_{\mathbb Z})$$
where the first isomorphism comes from,  by viewing $M=B(G\rtimes\mathbb Z)=E(G\rtimes\mathbb Z)/G\rtimes\mathbb Z$,  the induction structure of an equivariant homology, see \cite[Remark 61 on page 736]{LR},  and the second isomorphism is because the two homology theories $H_*(-; \textbf{L}^s_\mathbb Z), H_*(-; \textbf L^{<-\infty>}_\mathbb Z)$ are naturally isomorphic (since they are naturally isomorphic on a point). 
Now it is a fact from surgery theory that $p_i$ is injective when $i=n$, and bijective when $i>n$. 
This implies that  $\sigma_n$ is injective and $\sigma_{n+1}$ is bijective. Therefore $S^s(M)$ is trivial and the simple Borel conjecture holds for $G\rtimes\mathbb Z$. 

If in addition $Wh(G\times\mathbb Z)=0$, then by Lemma \ref{l group}, we have 
\begin{align}\label{sh-infinity}
L^h_n(\mathbb Z[G\rtimes\mathbb Z])\cong L^{<-\infty>}_n(\mathbb Z[G\rtimes\mathbb Z]), \forall n\in\mathbb Z
\end{align}
This implies, by using the surgery long exact sequence for the structure set $S^h(M)$, which consists of $h$-cobordant equivalence classes of all homotopy equivalent maps to $M$,   and a similar argument as above, that $S^h(M)$ consists of a single element.  This implies the bordism Borel conjecture holds for $G\rtimes\mathbb Z$.  This completes the proof of Theorem A.

\frenchspacing

\bibliographystyle{plain}
\bibliography{kun}

\begin{thebibliography}{10}

\bibitem{ACFP}
Douglas~R. Anderson, Francis~X. Connolly, Steven~C. Ferry, and Erik~K.
  Pedersen.
\newblock Algebraic {$K$}-theory with continuous control at infinity.
\newblock {\em J. Pure Appl. Algebra}, 94(1):25--47, 1994.

\bibitem{BFL}
A.~Bartels, F.~T. Farrell, and W.~L{\"u}ck.
\newblock The {F}arrell-{J}ones conjecture for cocompact lattices in virtually
  connected {L}ie groups.
\newblock {\em J. Amer. Math. Soc.}, 27(2):339--388, 2014.

\bibitem{BEL}
Arthur Bartels, Siegfried Echterhoff, and Wolfgang L{\"u}ck.
\newblock Inheritance of isomorphism conjectures under colimits.
\newblock In {\em {$K$}-theory and noncommutative geometry}, EMS Ser. Congr.
  Rep., pages 41--70. Eur. Math. Soc., Z\"urich, 2008.

\bibitem{BFJR}
Arthur Bartels, Tom Farrell, Lowell Jones, and Holger Reich.
\newblock On the isomorphism conjecture in algebraic {$K$}-theory.
\newblock {\em Topology}, 43(1):157--213, 2004.

\bibitem{BL2}
Arthur Bartels and Wolfgang L{\"u}ck.
\newblock The {B}orel conjecture for hyperbolic and {${\rm CAT}(0)$}-groups.
\newblock {\em Ann. of Math. (2)}, 175(2):631--689, 2012.

\bibitem{BLR2}
Arthur Bartels, Wolfgang L{\"u}ck, and Holger Reich.
\newblock The {$K$}-theoretic {F}arrell-{J}ones conjecture for hyperbolic
  groups.
\newblock {\em Invent. Math.}, 172(1):29--70, 2008.

\bibitem{BLRR}
Arthur Bartels, Wolfgang L{\"u}ck, Holger Reich, and Henrik R{\"u}ping.
\newblock K- and {L}-theory of group rings over {$GL_n({\bf Z})$}.
\newblock {\em Publ. Math. Inst. Hautes \'Etudes Sci.}, 119:97--125, 2014.

\bibitem{BR}
Arthur Bartels and Holger Reich.
\newblock Coefficients for the {F}arrell-{J}ones conjecture.
\newblock {\em Adv. Math.}, 209(1):337--362, 2007.

\bibitem{CP}
M.~C{\'a}rdenas and E.~K. Pedersen.
\newblock On the {K}aroubi filtration of a category.
\newblock {\em $K$-Theory}, 12(2):165--191, 1997.

\bibitem{Chap}
T.~A. Chapman.
\newblock Topological invariance of {W}hitehead torsion.
\newblock {\em Amer. J. Math.}, 96:488--497, 1974.

\bibitem{DL}
James~F. Davis and Wolfgang L{\"u}ck.
\newblock Spaces over a category and assembly maps in isomorphism conjectures
  in {$K$}- and {$L$}-theory.
\newblock {\em $K$-Theory}, 15(3):201--252, 1998.

\bibitem{DM}
Michael~W. Davis.
\newblock Poincar\'e duality groups.
\newblock In {\em Surveys on surgery theory, {V}ol. 1}, volume 145 of {\em Ann.
  of Math. Stud.}, pages 167--193. Princeton Univ. Press, Princeton, NJ, 2000.

\bibitem{fj}
F.~T. Farrell and L.~E. Jones.
\newblock Isomorphism conjectures in algebraic {$K$}-theory.
\newblock {\em J. Amer. Math. Soc.}, 6(2):249--297, 1993.

\bibitem{FW3}
F.~Thomas Farrell and Xiaolei Wu.
\newblock The {F}arrell-{J}ones conjecture for some nearly crystallographic
  groups.
\newblock {\em Algebr. Geom. Topol.}, 15(3):1667--1690, 2015.

\bibitem{FW2}
F.~Thomas Farrell and Xiaolei Wu.
\newblock Isomorphism conjecture for {B}aumslag-{S}olitar groups.
\newblock {\em Proc. Amer. Math. Soc.}, 143(8):3401--3406, 2015.

\bibitem{FW1}
Tom Farrell and Xiaolei Wu.
\newblock The {F}arrell-{J}ones conjecture for the solvable
  {B}aumslag-{S}olitar groups.
\newblock {\em Math. Ann.}, 359(3-4):839--862, 2014.

\bibitem{FRR}
Steven~C. Ferry, Andrew Ranicki, and Jonathan Rosenberg.
\newblock A history and survey of the {N}ovikov conjecture.
\newblock In {\em Novikov conjectures, index theorems and rigidity, {V}ol.\ 1
  ({O}berwolfach, 1993)}, volume 226 of {\em London Math. Soc. Lecture Note
  Ser.}, pages 7--66. Cambridge Univ. Press, Cambridge, 1995.

\bibitem{GMR}
Giovanni Gandini, Sebastian Meinert, and Henrik R\"uping.
\newblock The {F}arrell-{J}ones conjecture for fundamental groups of graphs of
  abelian groups.
\newblock {\em Groups, Geometry, and Dynamics}, 9(3):783--792, 2015.

\bibitem{HPR1}
Ian Hambleton, Erik~K. Pedersen, and David Rosenthal.
\newblock Assembly maps for group extensions in {$K$}-theory and {$L$}-theory
  with twisted coefficients.
\newblock {\em Pure Appl. Math. Q.}, 8(1):175--197, 2012.

\bibitem{HLS}
N.~Higson, V.~Lafforgue, and G.~Skandalis.
\newblock Counterexamples to the {B}aum-{C}onnes conjecture.
\newblock {\em Geom. Funct. Anal.}, 12(2):330--354, 2002.

\bibitem{HPR}
Nigel Higson, Erik~Kj{\ae}r Pedersen, and John Roe.
\newblock {$C^\ast$}-algebras and controlled topology.
\newblock {\em {$K$}-Theory}, 11(3):209--239, 1997.

\bibitem{KLR}
Holger Kammeyer, Wolfgang L\"uck, and Henrik R\"uping.
\newblock The {F}arrell-{J}ones conjecture for arbitrary lattices in virtually
  connected {L}ie groups.
\newblock {\em to appear in Geometry and Topology, arXiv:1401.0876}.

\bibitem{Luck1}
Wolfgang L{\"u}ck.
\newblock {$K$}- and {$L$}-theory of the semi-direct product of the discrete
  3-dimensional {H}eisenberg group by {${\mathbb Z}/4$}.
\newblock {\em Geom. Topol.}, 9:1639--1676 (electronic), 2005.

\bibitem{Luck2}
Wolfgang L{\"u}ck.
\newblock Survey on aspherical manifolds.
\newblock {\em European Congress of Mathematics}, pages 53--82, 2010.

\bibitem{LR}
Wolfgang L{\"u}ck and Holger Reich.
\newblock The {B}aum-{C}onnes and the {F}arrell-{J}ones conjectures in {$K$}-
  and {$L$}-theory.
\newblock {\em Handbook of {$K$}-theory. Vol. 1, 2}, pages 703--842, 2005.

\bibitem{LS}
Wolfgang L{\"u}ck and Wolfgang Steimle.
\newblock A twisted {B}ass-{H}eller-{S}wan decomposition for the non-connective
  {$K$}-theory of additive categories.
\newblock {\em to appear in Forum Mathematicum, arXiv: 1309.1353}.

\bibitem{RAA1}
A.~A. Ranicki.
\newblock Algebraic {$L$}-theory. {III}. {T}wisted {L}aurent extensions.
\newblock pages 412--463. Lecture Notes in Mathematics, Vol. 343. Springer,
  Berlin, Proc. Conf. Seattle Res. Center, Battelle Memorial Inst., 1972.

\bibitem{RH}
Henrik R\"uping.
\newblock The {F}arrell-{J}ones conjecture for {$S$}-arithmetic groups.
\newblock {\em arXiv:1309.7236}.

\bibitem{Wall3}
C.~T.~C. Wall.
\newblock List of problems.
\newblock volume~36 of {\em London Math. Soc. Lecture Note Ser.}, pages
  369--394. Cambridge Univ. Press, Cambridge-New York, Proc. Sympos., Durham,
  1977.

\bibitem{WC1}
Christian Wegner.
\newblock The {F}arrell-{J}ones conjecture for virtually solvable groups.
\newblock {\em to appear in Journal of Topology, arXiv:1308.2432}.

\bibitem{WC}
Christian Wegner.
\newblock The {$K$}-theoretic {F}arrell-{J}ones conjecture for {CAT}(0)-groups.
\newblock {\em Proc. Amer. Math. Soc.}, 140(3):779--793, 2012.

\bibitem{YGL}
Guoliang Yu.
\newblock The coarse {B}aum-{C}onnes conjecture for spaces which admit a
  uniform embedding into {H}ilbert space.
\newblock {\em Invent. Math.}, 139(1):201--240, 2000.

\end{thebibliography}

\end{document}